\documentclass[a4paper,11pt]{article}
\usepackage{amsmath,amsthm,amssymb,amsfonts,graphicx,hyperref}
\usepackage[top=3cm, bottom=3cm, left=3cm, right=3cm]{geometry}
\usepackage[shortlabels]{enumitem}
\usepackage{booktabs}
\usepackage{tikz}
\usetikzlibrary{decorations.pathmorphing}

\makeatletter
\renewcommand*{\@fnsymbol}[1]{\ensuremath{\ifcase#1\or \S\or \dagger\or \ddagger\or
   \mathsection\or \mathparagraph\or \|\or **\or \dagger\dagger
   \or \ddagger\ddagger \else\@ctrerr\fi}}
\makeatother

\newcommand{\R}{\mathbb{R}}

\newcommand{\nn}{\ensuremath{\mathcal{N}}}
\newcommand{\sym}{\ensuremath{\mathcal{S}}}
\newcommand{\dnn}{\ensuremath{\mathcal{DN}}}
\newcommand{\psd}{\ensuremath{\mathcal{PSD}}}
\newcommand{\spn}{\ensuremath{\mathcal{SPN}}}
\newcommand{\cop}{\ensuremath{\mathcal{COP}}}
\newcommand{\cp}{\ensuremath{\mathcal{CP}}}

\def\x{\mathbf x}
\def\y{\mathbf y}
\def\z{\mathbf z}
\def\u{\mathbf u}
\def\v{\mathbf v}
\def\d{\mathbf d}
\def\w{\mathbf w}

\def\e{\mathbf e}
\def\b{\mathbf b}
\def\c{\mathbf c}
\def\a{{\mathbf a}}

\newcommand{\1}{\mathbf 1}
\newcommand{\0}{\mathbf 0}

\newtheorem{defin}{Definition}[section]

\newtheorem{theorem}[defin]{Theorem}
\newtheorem{remark}[defin]{Remark}
\newtheorem{corollary}[defin]{Corollary}
\newtheorem{lemma}[defin]{Lemma}

\newtheorem{conjecture}[defin]{Conjecture}

\DeclareMathOperator{\trace}{trace}

\DeclareMathOperator{\supp}{supp}

\DeclareMathOperator{\diag}{diag}

\newcounter{figno}
\DeclareRobustCommand{\fig}[1]{\refstepcounter{figno}Figure \thefigno\label{#1}}

\def\ignore#1{}
\def\cob#1{#1}
\date{\today}

\title{SPN graphs: when copositive $=$ SPN\thanks{This work was supported by grant no.\ G-18-304.2/2011 by the German-Israeli Foundation for Scientific Research and Development (GIF). }
}

\author{
Naomi Shaked-Monderer\thanks{The Max Stern Yezreel Valley College, Yezreel Valley 19300, Israel.
Email: nomi@tx.technion.ac.il }
}

\begin{document}
\maketitle
\begin{abstract}
\noindent A real symmetric  matrix $A$ is  {copositive} if $\x^TA\x\ge 0$ for every nonnegative vector $\x$.
A matrix  is {SPN} if it is a sum of a real positive semidefinite matrix and a nonnegative one.
Every SPN matrix is copositive, but the converse does not hold for matrices of order greater than $4$.
A graph $G$ is an SPN graph if every copositive matrix whose graph is $G$ is SPN.
In this paper we present sufficient conditions for a graph to be SPN (in terms of its possible blocks)
and necessary conditions for a graph  to be SPN (in terms of  forbidden subgraphs). We also discuss the remaining gap
between these two sets of conditions, and make a conjecture regarding the complete characterization of SPN graphs.

\medskip



\end{abstract}

\section{Introduction}
A real symmetric matrix $A$ is  \emph{copositive} if $\x^TA\x\ge 0$ for every nonnegative vector $\x$.
A matrix  is  \emph{SPN} if it is a sum of a real positive semidefinite matrix and a nonnegative one.
The set of copositive matrices of order $n$ is denoted by $\cop_n$, and the set of SPN matrices of that order is
denoted by $\spn_n$.  For every $n$, $\spn_n\subseteq \cop_n$, and it is known that for $n\le 4$ these sets are equal,
but  for $n\ge 5$ the inclusion is strict. That is, being SPN is a sufficient
condition for copositivity, but not a necessary one.

We consider the following question: For which zero-nonzero  patterns of symmetric matrices the sufficient condition for copositivity of being SPN is
also a necessary condition? The zero-nonzero pattern of a symmetric  $n\times n$ matrix is described, as usual, by
its graph $G(A)$. The vertices of $G(A)$ are $\{1,\dots, n\}$, and $ij$ is an edge if and only if $a_{ij}\ne 0$. We say that a graph $G$ is  \emph{SPN} if every copositive matrix with graph $G$ is SPN, and consider the problem of characterizing all SPN graphs.

Copositive matrices  have been studied since the 1950's.
In recent years there is an increased interest in this class of matrices in the field of
mathematical optimization, see surveys  \cite{Duer2010, Bomze2012}.  It has been shown that many combinatorial and nonconvex quadratic optimization problems
can be formulated as linear problems over the cone $\cop_n$.
In this formulation, the difficulty of the problem lies entirely in the cone $\cop_n$, as all the other constraints are linear. And difficult it is. It has been proved in~\cite{MurtyKabadi1987} that checking whether a given matrix is copositive is a co-NP-complete problem.
Checking whether a given matrix is SPN, on the other hand,  is considered tractable; it can be  done by a semidefinite program.

The problem of studying SPN graphs is motivated not only by the difficulty of checking copositivity, but also by the related characterization of completely positive graphs.
A matrix $A$ is \emph{completely positive} if $A=BB^T$, where $B$ is a nonnegative, not necessarily square, matrix. A matrix $A$ is \emph{doubly nonnegative}
if $A$ is both positive semidefinite and entrywise nonnegative. The set of completely positive $n\times n$ matrices is denoted by $\cp_n$, and that of doubly nonnegative $n\times n$ matrices by $\dnn_n$. Each of the four sets $\cop_n$, $\spn_n$, $\cp_n$ and $\dnn_n$ is a proper cone in the space
$\sym_n$ of real symmetric $n\times n$ matrices.
The cones $\cp_n$ and $\dnn_n$ are the dual cones of $\cop_n$
and $\spn_n$, respectively, with respect to the inner product $\langle A,B\rangle=\trace(AB)$.
For every $n$ we have $\cp_n\subseteq \dnn_n$, and equality holds only for $n\le 4$. A graph $G$ is said to be
 \emph{completely positive} if every doubly nonnegative matrix with graph $G$ is completely positive. Completely positive
graphs were fully characterized in a series of papers, see the monograph \cite{BermanShaked2003} and the references therein.
These are all the graphs that do not contain an odd cycle
of length greater or equal $5$. Equivalently, these are graphs in which each block either has at most $4$ vertices, or is bipartite, or consists
of several triangles sharing a common base.
We are looking for a similar characterization of SPN graphs.
Note that in a doubly nonnegative matrix  every entry is either positive or zero, and thus the graph of the matrix captures
the sign pattern of its off-diagonal entries completely. In contrast, a copositive matrix may have also negative off-diagonal elements.
It therefore makes sense to consider a signed graph to describe the sign pattern of such a matrix. A \emph{signed graph} is a graph with an additional
assignment of  plus or minus sign  to each edge. With each symmetric matrix $A$ we associate its  signed graph $\mathcal{G}(A)$, which
has the same vertices and edges as $G(A)$,  with additional assignment of signs to
the edges. An edge in $\mathcal{G}(A)$ is positive if $a_{ij}>0$  and  negative if $a_{ij}<0$.
We say that a signed graph $\mathcal{G}$ is \emph{SPN} if every copositive matrix whose signed graph is $\mathcal{G}$ is SPN.
Characterizing SPN signed graphs seems to be a natural problem,
but here we will touch it only insofar as  it serves  the characterization of SPN graphs.

In this paper we find some sufficient conditions for a graph to be SPN  in terms of its possible blocks, and some necessary conditions in terms of forbidden subgraphs.
We also discuss the gap between the necessary and sufficient conditions, and make a couple of conjectures regarding the complete solution.
In the process we prove several results that may be of interest in their own right. One outcome is a new proof from general principles for the equality $\cop_n=\spn_n$, $n\le 4$.

The paper is organized as follows. In Section 2 we discuss terminology, notations and known results.
In Section 3 we prove results that will be needed in the characterization, some are adaptations or generalizations
of known results, the other are new.
In Section 4 we prove some basic results on SPN graphs.
In Section 5 we use graphs and some general results, to present a new proof  that $\cop_4=\spn_4$, and characterize copositive matrices with acyclic graph.
In Section 6 we prove that any graph consisting of several triangles with a common base is an SPN graph.
In Section 7 we fully characterize the SPN graphs on $5$ vertices.
In Section 8 we discuss  the effect of certain graph transformations on signed graphs on the property of being SPN/non-SPN.
In Section 9 we use all the results to give some sufficient conditions for a graph to be SPN, in terms of its possible blocks, and some necessary
conditions for a graph to be SPN, in terms of forbidden subgraphs.
We conclude with a discussion of the remaining gap between these two sets of conditions, and with \cob{some} conjectures regarding
the complete characterization of SPN graphs.

\section{Preliminaries}
\subsection{Additional notations and terminology}
We use $\ge $ to denote entry-wise inequality (between vectors or matrices), and $>$ when in each entry the inequality is strict.
Vectors in $\R^n$ are column vectors, $\R^n_+$ is the cone of nonnegative vectors in $\R^n$, and $\R^n_{++}$ is its interior, consisting of positive vectors.
The norm on $\R^n$ is the Euclidean norm $\|\x\|=\sqrt{\x^T\x}$.
A vector of all ones is denoted by $\1$, a zero vector by $\0$. The standard basis vectors in $\R^n$ are denoted by $\e_1, \dots, \e_n$.
For $\x\in \R^n$, $\supp \x=\{i\in\{1,\dots, n\}\, |\, x_i\ne 0\}$.
The \cob{maximum}\ignore{minimum} of two vectors $\cob{\max}(\x,\y)$ is computed entrywise, and for $\x\in \R^n$, $\x_+=\cob{\max}(\x,\0)$ and $\x_-=\cob{\max}(-\x,\0)$ (so
$\x=\x_+-\x_-$). For $\x\in \R^n$ and $\alpha\subseteq \{1,\dots, n\}$, $\x[\alpha]$ denotes the vector of length $|\alpha|$ consisting of
entries of $\x$ indexed by $\alpha$.

A matrix of all ones is denoted by $J$, and $E_{ij}\in \sym_n$ is defined by
\[E_{ij}=\begin{cases}\e_i\e_i^T\qquad& \text{if }i=j \\
\e_i\e_j^T+\cob{\e_j\e_i^T}\quad &\text{if } 1\le i\ne j\le n\end{cases}\]
The vector of diagonal entries of a square matrix $A$ is denoted by $\diag A$.
For $A\in \sym_n$  and
$\alpha \subseteq\{1, \ldots, n\}$, $A[\alpha]$  is the principal submatrix whose rows and columns are indexed by $\alpha$ and  $A(\alpha)$ is the principal submatrix whose
rows and columns are indexed by the complement $\alpha^c=\{1,\dots, n\}\setminus \alpha$. We abbreviate and write $A[1,\dots, m]$ instead of $A[\{1, \dots, m\}]$   and $A(1,\dots, m)$ instead of $A(\{1, \dots, m\})$. A direct sum of matrices is denoted by $\oplus$.

The set of nonnegative matrices in $\sym_n$ is denoted by $\nn_n$, and the
set of positive semidefinite matrices in $\sym_n$ is denoted  by $\psd_n$. Like $\cop_n$, $\spn_n$, $\cp_n$ and $\dnn_n$, each of the sets $\nn_n$ and $\psd_n$ is
a proper cone in $\sym_n$ (that is, a closed convex cone, which is pointed and has an nonempty interior).
Clearly, $\psd_n+\nn_n=\spn_n$. When discussing the space of symmetric matrices or one of these cones  we often omit the order $n$ from the notation.

In our graph theoretic terminology and notations we mostly follow~\cite{Diestel2010}. For completeness we recall here
some terms concerning subgraphs and paths.
We consider only \emph{simple graphs}, i.e., undirected graphs without multiple edges or loops.
The vertex set of a graph $G$ is denoted by $V(G)$ and its edge set by $E(G)$.
A graph $H$ is a \emph{subgraph} of $G$ if $V(H)\subseteq V(G)$ and $E(H)\subseteq E(G)$. The subgraph is \emph{proper} if
one of the inclusions is strict. The subgraph $H$ is \emph{induced}
if $E(H)$ consists of all the edges of $G$ which have both ends in $V(H)$.
A subgraph $H$  of $G$ is \emph{spanning} if $V(H)=V(G)$. If $e$ is an edge of a graph $G$, $G-e$ denotes the subgraph obtained by omitting the edge $e$.
We denote an edge $\{x,y\}$ by $xy$. The edges of a path $P$ are emphasized by dashes: $P=v_0-v_1-\dots - v_k$, where $v_0,\dots,v_k$ are $k+1$ distinct vertices.  \cob{The vertices} $v_0$ and $v_k$ are the \emph{ends} of the path $P$, and $v_1,\dots, v_{k-1}$ are its \emph{inner} vertices. The \emph{length} of $P$ is the number of its edges. If $v_0=x$ and $v_k=y$ we say that $P$ is
an \emph{$x-y$ path} ($P$ \emph{links} $x$ and $y$). The \emph{distance} in $G$ between vertex $x$
and vertex $y$ is    denoted by $d_G(x,y)$. It is the length of the shortest $x-y$ path ($d_G(x,y)=-\infty$ if there is no $x-y$ path).
For a subgraph $H$ of $G$, an \emph{$H$-path} is a path linking two different vertices in $H$, whose inner vertices are not in $H$.
Paths are \emph{independent} if they do not have inner vertices in common. We will
say that a single path $P$ is \emph{independent} if no inner vertex of $P$ has a neighbor in $V(G)\setminus V(P)$.

In drawings of signed graphs, dashed line are used to denote  negative edges, while solid lines denote positive edges.

\subsection{Background}
We first mention some elementary facts about the classes of matrices  in question. Let $\mathcal{C}$ denote one of these classes
of matrices: real positive semidefinite, symmetric nonnegative, copositive, SPN. If $A$ is in the class $\mathcal{C}$, then every
principal submatrix of $A$ is also in $\mathcal{C}$. If $A=A_1\oplus A_2$ is a real symmetric matrix, then $A$ is in $\mathcal{C}$
if and only if $A_1$ and $A_2$ are both in $\mathcal{C}$. Therefore when convenient  we may restrict our attention to irreducible symmetric matrices
and to connected graphs.
The convex cones of matrices in these classes of a given order, and in particular $\cop$ and $\spn$, are closed under permutation similarity
($A \mapsto P^TAP$, where $P$ is a permutation matrix)
and under positive diagonal congruence ($A \mapsto DAD$, where $D$ is a positive diagonal matrix).
Thus we often permute rows and columns of a matrix simultaneously for convenience, or use
diagonal scaling to replace a given matrix with a positive diagonal by a matrix with diagonal $\1$.
Let $\mathcal{K}$ be one of the cones  $\cop$ or $\spn$. If $A\ge B$, where $B$ is in $\mathcal{K}$, then
$A$ is also in $\mathcal{K}$. Another simple observation is that any SPN matrix $A$  has a decomposition $A=P+N$,
where $P\in \psd$, $N\in \nn$, $\diag P=\diag A$ and $\diag N=\0$.

The Schur complement is an important tool in identifying positive semidefinite matrices.
If
\[A=\left(\begin{array}{cc}
            M & E \\
            E^T & B
          \end{array}
\right)\in \sym,\]
and $M$ is nonsingular, the Schur complement of $M$ in $A$ is defined by
\[A/M=B-E^TM^{-1}E.\]
When the submatrix $M$ is a singular positive semidefinite matrix, we use the  generalized  Schur complement defined by
\[A/M=B-E^TM^\dagger E,\]
where $M^\dagger$ is the Moore-Penrose generalized inverse of $M$.
We refer to \cite{Zhang2005} for details on the Schur complement, and mention here only one basic result\cob{, see \cite[Theorem 1.20]{Zhang2005}}.

\begin{lemma}\label{lem:psdbySchur}
Let \[A=\left(\begin{array}{cc}
            M & E \\
            E^T & B
          \end{array}
\right)\in \sym.\]
Then $A\in \psd$ if and only if $M\in \psd$, the column space of $E$ is contained in that of $M$, and $A/M\in \psd$.
\end{lemma}

Unfortunately, there is no parallel result for identifying copositive matrices (but see Lemmas \ref{lem:+row-row} and \ref{lem:M+-rows} in
the next section).

There are a handful of results about characterizing copositivity. We mention here
Kaplan's copositivity criteria, which reflects the complexity of checking copositivity.

\begin{lemma}\cite[Theorem 2]{Kaplan2000}\label{lem:Kaplan}
A matrix $A\in \sym_n$ is copositive if and only if no principal submatrix of $A$ has a positive eigenvector
corresponding to a negative eigenvalue.
\end{lemma}

We state the next few results for matrices with diagonal equal to $\1$. For such matrices Hoffman and Pereira proved the following.

\begin{lemma} \cite[Lemma 3.1]{HoffmanPereira1973}\label{lem:diag1 cop}
A matrix $A\in \sym$ with $\diag A =\1$ is copositive if and only if each principal submatrix of $A$ in which
all the off-diagonal entries are less than $1$ is copositive.
\end{lemma}

They also gave a full characterization of copositive matrices with entries in $\{0, 1, -1\}$.  The characterization
is stated in terms of the graph $G_{-1}(A)$ of $A\in \sym_n$, defined as follows: its vertices are $\{1,\dots, n\}$, and $ij$ is an edge if and only if $a_{ij}=-1$.

\begin{lemma} \cite[Theorem 3.2]{HoffmanPereira1973}\label{lem:01-1cop}
Let $A\in \sym$ have $\diag A =\1$, and  $a_{ij}\in\{0, 1, -1\}$  for every $i, j$. Then $A$ is copositive if and only if the graph $G_{-1}(A)$ is triangle
free, and $a_{ij}=1$ for every $i\ne j$ such that $d_{G_{-1}(A)}(i,j)=2$.
\end{lemma}

SPN matrices with diagonal $\1$ and a connected $G_{-1}$ graph were fully characterized in \cite{ShakedBermanDuerKannan2014}. This result is stated next.

\begin{lemma}\cite[Lemma 3.5]{ShakedBermanDuerKannan2014}\label{lem:G-1SPN}
Let $A\in \sym$ have $\diag A={\1}$ and $a_{ij}\ge -1$ for every $i, j$, and let $G_{-1}(A)$ be connected. Then $A\in \spn$ if and only if the following two conditions are
satisfied: $G_{-1}(A)$ is bipartite
and $a_{ij}\ge 1$ whenever $d_{G_{-1}(A)}(i,j)$ is even.
\end{lemma}

If in the last lemma $A\in \cop$, then the lower bound on the entries is automatically satisfied  by the following basic fact
(which follows from the copositivity of the $2\times 2$ principal submatrices of $A$).

\begin{lemma}\cite[Lemma 2]{Diananda1962}\label{lem:offdiag}
Let  $A\in \cop$. Then $a_{ij}\ge -\sqrt{a_{ii}a_{jj}}$ for every $i\ne j$. In particular,
\begin{itemize}
\item[{\rm(a)}] If $a_{ii}=0$, then $a_{ij}\ge 0$ for every $j\ne i$.
\item[{\rm(b)}] If $\diag A =\1$, then $a_{ij}\ge -1$ for every $i\ne j$.
\end{itemize}
\end{lemma}

When $A\in \sym$ has all off-diagonal entries nonpositive (that is, $A$ is a \emph{$Z$-matrix}), copositivity is easily characterized.
The result below  was proved by Li and Feng (and appeared previously without proof  in \cite[Exercise 3.53]{Murty1988}).

\begin{lemma}\cite[Theorem 4]{LiFeng1993}\label{lem:Zmat}
Let $A\in \sym$ have all off-diagonal nonpositive. Then $A$ is copositive if and only if $A$ is positive semidefinite.
\end{lemma}

Recall that  a  matrix $A$ is called an \emph{$M$-matrix}  if   $A=sI-B$ for
some $B\ge 0$ and $s\ge \rho(B)$, where $\rho(B)$ is the Perron-Frobenius eigenvalue of $B$.
We refer to \cite{BermanPlemmons1994} for details about $M$-matrices, and recall in the next two lemmas just a few  of the   results for the symmetric
case. (Note that Lemma \ref{lem:Zmat} can  also be deduced from \cob{Kaplan's} copositivity criteria above, and part (b) of the next lemma.)

\begin{lemma}\label{lem:Zmatres}
Let $A\in \sym$ be a $Z$-matrix. Then
\begin{itemize}
\item[{\rm(a)}] $A$ has a nonnegative eigenvector corresponding to its smallest eigenvalue.
If $A$ is irreducible, the smallest eigenvalue is simple and there exists a positive eigenvector corresponding to it.
\item[{\rm(b)}] $A$ is positive semidefinite if and only if $A$ is an $M$-matrix.
\end{itemize}
\end{lemma}

\begin{lemma}\label{lem:Mmat}
If $A\in \sym$ is a nonsingular $M$-matrix, then
\begin{itemize}
\item[{\rm(a)}] $A^{-1}\ge 0$. If $A$ is irreducible, then $A^{-1}>0$.
\item[{\rm(b)}] If  $B\ge A$ is a $Z$-matrix, then $B$ is a nonsingular $M$-matrix and $A^{-1}\ge B^{-1}$.
\end{itemize}
\end{lemma}

Following \cite{DickinsonDuerGijbenHildebrand2013a}, we say that $A\in \cop$ is \emph{irreducible with respect to $E_{ij}$}, or \emph{$\{i,j\}$-irreducible},
if for every $\delta >0$ the matrix $A-\delta E_{ij}\notin \cop$. We say that $A\in \cop$ is \emph{$\widetilde{\nn}$-irreducible} if $A$ is $\{i,j\}$-irreducible
for every $i\ne j$. (These notions should not be confused with irreducibility of the matrix!)  Obviously, to prove that $A\in \cop$ is SPN, it
suffices to show that some  $\{i,j\}$-irreducible $B=A-\delta E_{ij}$, or an $\widetilde{\nn}$-irreducible $B\le A$, is SPN.
In \cite{DickinsonDuerGijbenHildebrand2013a} $\{i,j\}$-irreducibility of a copositive matrix was characterized in terms the zeros of the matrix.
A vector $\u\in \R^n_+$ is \emph{zero} of $A\in \cop_n$ if $\u\ne \0$ and $\u^TA\u=0$. The set of all zeros of a copositive matrix $A$ is
denoted by
\[\mathcal{V}^A=\{\u\in  \R^n_+\cob{\,|\,} \u\ne \0 \text{ and } \u^TA\u=0.\}\]
Note that $\mathcal{V}^A$ is closed under multiplication by a positive scalar.
The following basic facts about zeros of a copositive matrix  date back to \cite[Lemma 7]{Diananda1962} (part(a)) and \cite[page 200]{Baumert1966}(parts (b) and (c)).

\begin{lemma}\label{lem:A[suppu]}
Let $A\in \cop$ and let $\u\in \mathcal{V}^A$. Then
\begin{itemize}
\item[{\rm(a)}] $A[\supp \u]$ is positive semidefinite.
\item[{\rm(b)}] $A\u\ge \0$.
\item[{\rm(c)}] $(A\u)_i=0$ for $i\in \supp \u$.
\end{itemize}
\end{lemma}

Recall that if $P\in \psd$ then $\u\in \mathcal{V}^P$ if and only if $P\u=\0$ (i.e., $\u$ is an eigenvector of $P$ corresponding to the
eigenvalue $0$).

We will use the following  characterization of $\{i,j\}$-irreducibility by zeros from \cite{DickinsonDuerGijbenHildebrand2013a}.
The case $i=j$ was proved in \cite{Baumert1966}, and
the case $i\ne j$  in \cite{DickinsonDuerGijbenHildebrand2013a}.

\begin{lemma}\cite[Theorem 2.6]{DickinsonDuerGijbenHildebrand2013a}\label{lem:ij irred}
Let $A\in \cop_n$, $n\ge 2$, and $1\le i, j\le n$. Then the following conditions are equivalent:
\begin{itemize}
\item[{\rm(a)}] $A$ is $\{i,j\}$-irreducible.
\item[{\rm(b)}]  There exists $\u\in \mathcal{V}^A$ such that $u_i+u_j>0$ and $(A\u)_i=(A\u)_j=0$.
\end{itemize}
\end{lemma}

\section{Some basic ideas}
In this section we prove some results that are useful for identifying SPN/copositive matrices and SPN graphs.

The first lemma is essentially Lemma 4.12 in \cite{DickinsonDuerGijbenHildebrand2013a}.
That lemma refers to $\widetilde\nn$-irreducibile matrices. Here we do not assume $\widetilde\nn$-irreducibility,
and the outcome changes accordingly.
We include  the proof,
since its details  are important in later results.

\begin{lemma}\label{lem:n-1psd}
Let $A\in \cop_n$ have a positive semidefinite submatrix of order $n-1$. Then $A\in \spn_n$.
\end{lemma}

\begin{proof}
Let
\[ A=\left(
      \begin{array}{cc}
        A_0 & \a \\
        \a^T & a_{nn} \\
      \end{array}
    \right), \quad
\] where  $\a\in \R^{n-1}$ and $A_0=A[1,\dots, n-1]$ is positive semidefinite. Let $c\ge 0$ be the minimal value for which the matrix
\[B=\left(
      \begin{array}{cc}
        A_0 & \a \\
        \a^T & c \\
      \end{array}
    \right)
\] is copositive. Then $B$ is $\{n,n\}$-irreducible, and thus
there exists a zero $\u$ of $B$ such that $u_n=1$. That is $\u^T=\left(\x^T ~ 1\right)$, $\x\in \R^{n-1}_+$. Then
 $(B\u)_n=0$ and $B\u\ge \0$ imply that
$\a^T\x+c=0$ and $A_0\x+\a\ge \0$. Thus
\[B=\left(
      \begin{array}{cc}
       A_0 & -A_0\x \\
        -\x^TA_0 & \x^TA_0\x \\
      \end{array}
    \right)+\left(
      \begin{array}{cc}
        0 & A_0\x+\a \\
        \x^TA_0+\a^T & c-\x^TA_0\x \\
      \end{array}
    \right) \]
   Since $\u^TB\u=0$ we have $\x^TA_0\x+2\a^T\x+c=0$, which together with $\a^T\x=-c$ implies that $\x^TA_0\x=c$. Thus $B$, and therefore $A$, is
    SPN. Note that for every $1\le i\le n-1$ such that $x_i>0$, \cob{$(B\u)_i=0$ by Lemma \ref{lem:A[suppu]}(c), so that} $(-A_0\x)_i=\cob{a_{in}}$.
\end{proof}

Similarly, the next lemma is an adaptation of  Corollary 4.14 in \cite{DickinsonDuerGijbenHildebrand2013a}.

\begin{lemma}\label{lem:suppu>=n-2+}
Let $A\in \cop_n$, and let $\u\in \mathcal{V}^A$ satisfy $|\supp\u|\ge n-2$ and $|\{i\, |\, (A\u)_i=0\}|\ge n-1$. Then $A$ is SPN.
\end{lemma}

\begin{proof}
The proof is identical to the proof of that
corollary,  except that Lemma \ref{lem:n-1psd} above should be used instead of Lemma 4.12 of \cite{DickinsonDuerGijbenHildebrand2013a}.
\end{proof}

Our next lemma uses Schur complements to identify a SPN/copositive matrix. For copositive matrices this result is known \cite{LiFeng1993} (with roots dating back to \cite{Bomze1987}). It was used in \cite{Bomze2000} to suggest an algorithm for testing copositivity of a tridiagonal matrix. This algorithm
was extended to matrices whose graph is acyclic  in \cite{Ikramov2002}. See \cob{Lemma \ref{lem:M+-rows}} for the proof.

\begin{lemma}\label{lem:+row-row}
Let
\begin{equation}A=\left(\begin{array}{cc}
            c & \a ^T\\
            \a & B
          \end{array}
\right),  \quad  c\ge  0~~,\a\in \R^{n-1}. \label{eq:caB}\end{equation}
\begin{itemize}
\item[{\rm{(a)}}] If $\a\ge \0$, then $A\in \spn_n$ ($A\in \cop_n$) if and only if $B=A(1)\in \spn_{n-1}$ ($B=A(1)\in \cop_{n-1}$).
\item[{\rm{(b)}}] If $\a \le \0$ and $c>0$, then $A\in \spn_n$ ($A\in \cop_n$) if and only if and $A/A[1]\in\spn_{n-1}$ ($A/A[1]\in\cop_{n-1}$).
\end{itemize}
\end{lemma}

Instead of proving Lemma \ref{lem:+row-row} (for the SPN case), we prove the following (both for copositive matrices and SPN matrices).

\begin{lemma}\label{lem:M+-rows}
Let $A$ be in the form \begin{equation}A=\left(\begin{array}{cc}
            M & E\\
            E^T & B
       \end{array}\right), \label{eq:MEB}\end{equation}
where $M\in \sym_r$, $B\in \sym_{n-r}$  and $E\in \R^{r\times (n-r)}$.
\begin{itemize}
\item[{\rm{(a)}}] If $E\ge 0$ and $M$ is SPN (copositive), then  $A\in \spn_n$ ($A\in \cop_n$) if and only if $B\in \spn_{n-r}$ ($B\in \cop_{n-r}$).
\item[{\rm{(b)}}] If $E\le 0$ and $M$ is a nonsingular $M$-matrix, then $A\in \spn_n$ ($A\in \cop_n$) if and only if  $A/M\in \spn_{n-r}$ ($A/M\in \cop_{n-r}$).
\end{itemize}
\end{lemma}

\begin{proof}
(a):
One implication is trivial from the fact that a principal submatrix of an SPN (a copositive) matrix is SPN (copositive).
The converse follows from the inequality $A\ge M\oplus B$.

(b): We have:
\[A=\left(\begin{array}{cc}
            M & E\\
            E^T & E^TM^{-1} E
          \end{array}
\right)+ \left(\begin{array}{cc}
            0 & 0\\
            0 & A/M
          \end{array}
\right). \]  That is, $A$ is a sum of a positive semidefinite \cob{matrix} and the matrix $0\oplus (A/M)$. Thus  if $A/M$ is SPN (copositive) then
$A$ is SPN (copositive).

It remains to  prove  that $A\in \spn_n$ ($A\in \cop_n$) implies $A/M\in \spn_{n-r}$ ($A/M\in \cop_{n-r}$).

In the copositive case: For every $\x\in \R^r_+$ and $\y\in \R^{n-r}_+$,
\[\left(
    \begin{array}{cc}
      \x^T & \y^T \\
    \end{array}
  \right) A \left(
              \begin{array}{c}
                \x \\
                \y\\
              \end{array}
            \right)=\x^TM\x+2\y^TE^T\x+ \y^TB\y\ge 0\, .
\]
In particular, for every  $\y\in \R^{n-r}_+$, the inequality holds for $\x=-M^{-1}E\y\ge \0$. For such $\x$
\begin{eqnarray*}\x^TM\x+ 2\y^TE^T\x+\y^TB\y&=&\y^TE^TM^{-1} E\y-2\y^TE^TM^{-1} E\y+\y^TB\y\\&=&\y^T(B-E^TM^{-1} E) \y.\end{eqnarray*}
Thus, $\y^T(B-E^TM^{-1} E) \y\ge 0 $ for every $\y\in \R^{n-r}_+$, i.e., $A/M\in \cop_{n-r}$.

For the SPN case:  If $A$ is SPN, there exists a positive semidefinite
\[P=\left(\begin{array}{cc}
            N & F\\
            F^T & C
          \end{array}
\right) \]
such that $A\ge P$ and $\diag P =\diag A $.  In particular,\cob{ }$N\le M$ is also an $M$-matrix, and $F\le E\le 0$.
Then, (permute the rows and columns so that)  $N=N_1\oplus N_2$, where $N_1$ is the direct
sum of all singular irreducible blocks of $N$ and $N_2$ is nonsingular, and
\[P=\left(\begin{array}{ccc}
            N_1 & 0 & F_1 \\
            0 & N_2 & F_2 \\
            F_1^T &F_2^T & C
          \end{array}
\right).\]
Let $\v$ be a positive eigenvector of $N_1$ corresponding to zero.
Then $\u= ( \v^T ~ \0 )^T\in \mathcal{V}^P$, and  hence
$P\u=\0$, and in particular $F_1^T\v=\0$. Since $F_1^T\le 0$ and $\v>\0$, we get that $F_1=0$ and
\[P=\left(\begin{array}{ccc}
            N_1 & 0 & 0 \\
            0 & N_2 & F_2 \\
            0 &F_2^T & C
          \end{array}
\right).\]
As $M\ge N$  is an $M$-matrix, the corresponding block structure of $A$ is
\[A=\left(\begin{array}{ccc}
            M_1 & 0 &0 \\
            0 & M_2 & E_2 \\
            0 &E_2^T & B
          \end{array}
\right).\]
 Thus
$E^TM^{-1}E=0\oplus E_2^TM_2^{-1}E_2$. Since $N_2\le M_2$ are nonsingular $M$-matrices, we have $N_2^{-1}\ge M_2^{-1}$, and therefore
\[A/M=B-\cob{E_2^TM_2^{-1}E_2}\ge C-\cob{F_2^TN_2^{-1}F_2}.\]
The matrix on the right hand side is a principal submatrix of $P/N_2$, which is positive semidefinite,
and therefore $A/M$ is SPN.
\end{proof}

\begin{remark}\label{rem:singular M}{\rm
 If $A\in \cop$ is in the form  \eqref{eq:MEB},
and $M$ is a singular $M$-matrix, then an argument similar to the one at the end of the last proof shows that $A$ is reducible:
Let $M=M_1\oplus M_2$, where $M_1$ is a direct sum of all the singular blocks of $M$, and $M_2$ is nonsingular (though may be empty).
Accordingly,
\[A=\left(\begin{array}{ccc}
            M_1 & 0 &E_1 \\
            0 & M_2 & E_2 \\
            E_1^T &E_2^T & B
          \end{array}
\right).\]
Let $\v$ be a positive eigenvector of $M_1$ corresponding to
to the eigenvalue $0$. Then $\u= \left(\begin{array}{cc}
                                                                \v^T &
                                                                \0
                                                              \end{array}
                                                            \right)^T\in \mathcal{V}^A$,
and thus $A\u\ge \0$. Thus $E_1^T\v=\0$, which by the nonpositivity of $E_1$ and the positivity of $\v$ implies that $E_1=0$.
Thus in this case,
\[A=\left(\begin{array}{ccc}
            M_1 & 0 &0 \\
            0 & M_2 & E_2 \\
            0 &E_2^T & B
          \end{array}
\right).\]
Therefore, when $A\in \cop_n$ of the form \eqref{eq:MEB} has a connected graph, the nonsingularity
assumption required in part (b) of the previous lemma is automatically satisfied.
}\end{remark}

The proof of the next lemma uses the fact that the matrix $I-\x\x^T$ is positive semidefinite if and only if $||\x||\le 1$.

\begin{lemma}\label{lem:IvvT}
Let $\v\in \R^n$. Then the following are equivalent:
\begin{enumerate}
\item[\rm{(a)}] $I-\v\v^T\in \cop$
\item[\rm{(b)}] $I-\v\v^T\in\spn$
\item[\rm{(c)}] $||\v_-||\le 1$ and $||\v_+||\le 1$.
\end{enumerate}
\end{lemma}

\begin{proof} Let $A=I-\v\v^T $.\\
(c)$\Rightarrow $(b):
 If $||\v_-||\le 1$ and $||\v_+||\le 1$, then the submatrices $A[\sigma_-]$ and $A[\sigma_+]$, where $\sigma_-=\supp \v_-$ and
$\sigma_+=\supp\v_+$, are both positive semidefinite.  We have $A\ge P$, where   $P$ is a matrix of the same order of $A$, which is a direct sum of $A[\sigma_-]$,
$A[\sigma_+]$  and possibly also a zero matrix. Thus that $A\in \spn$.\\
(a)$\Rightarrow $(c): If $||\v_-||>1$, then the principal submatrix $A[\sigma_-]$  has a negative eigenvalue $1-||\v_-||^2$, with
a corresponding positive eigenvector $\v_-[\sigma_-]$.  Similarly, if   $||\v_+||>1$ then $A[\sigma_+]$  has a negative eigenvalue $1-||\v_+||^2$, with a corresponding positive
eigenvector $\v_+[\sigma_+]$. Thus by Kaplan's copositivity criteria, in both these cases $A$ is not copositive.
\end{proof}

With each $A\in \sym_n$ we associate, in addition to the graph $G(A)$ and the signed graph $\mathcal{G}(A)$, two more graphs: the graph $G_-(A)$ with
vertices $\{1, \dots, n\}$, where $ij$ is an edge if and only if $a_{ij}<0$, and
the graph $G_+(A)$
with vertices $\{1, \dots, n\}$, where $ij$ is an edge if and only if $a_{ij}>0$.  By definition, both $G_-(A)$ and $G_+(A)$ are spanning subgraphs of $G(A)$, $E(G(A))=E(G_-(A))\cup E(G_+(A))$, and the union is edge-disjoint.
The next lemma presents the role of $G_-(A)$ in determining whether a matrix is SPN/copositive. It is a simple observation, which is quite useful. 

\begin{lemma}\label{lem:A-}
A matrix $A\in \sym$ is SPN (copositive) if and only if for every  set $\alpha$, which is the vertex set of a connected component
of $G_-(A)$, $A[\alpha]$ is SPN (copositive).
\end{lemma}

\begin{proof}
The `only if' is clear from the fact that every principal submatrix of an SPN (copositive) matrix is SPN (copositive).
The reverse implication follows from the inequality
$A\ge A[\alpha_1]\oplus  \dots \oplus A[\alpha_k]$, where $\alpha_1, \dots, \alpha_k$ are the vertex sets of the connected components of $G_-(A)$.
\end{proof}

The next two results follow from  Lemma \ref{lem:A-}:

\begin{corollary}\label{cor:-induced}
If $A\in \cop$ and $G_-(A)$ is a disjoint union of induced subgraphs of $G(A)$, then $A\in \spn$.
\end{corollary}

\begin{proof}
In that case, for every set $\alpha$, which  is the vertex set of a connected component of $G_-(A)$, $A[\alpha]$ is a copositive $Z$-matrix
and is therefore positive semidefinite.
\end{proof}

\begin{corollary}\label{cor:minimlally non SPN}
If every proper subgraph of a graph $G$ is SPN, and $A\in \cop$ with $G(A)=G$ is not SPN, then
$G_-(A)$ is connected.
\end{corollary}

We conclude this section with  a result on copositive matrices whose graph has a cut vertex.

\begin{lemma}\label{lem:semidirectsum}
Let $A\in \cop$ have graph $G$, where $G=G_1\cup G_2$ and $G_1\cap G_2$ is a single vertex.
Then there exist copositive matrices $A_1$ and $A_2$ such that $G(A_1)=G_1$, $G(A_2)=G_2$ and $A=(A_1\oplus 0)+(0\oplus A_2)$.
\end{lemma}

\begin{proof}
We may assume that  $V(G_1)=\{1,\dots, k\}$ and $V(G_2)=\{k, \dots, n\}$. It suffices to show that if $A$ is $\{k,k\}$-irreducible
then such $A_1$ and $A_2$ exist. Let
\[A=\left(\begin{array}{ccc}
            B_1 & \b & 0 \\
            \b^T & a & \c^T \\
            0 & \c & B_2
          \end{array}
\right),~~ B_1\in \cop_{k-1}, ~B_2\in \cop_{n-k} ~, ~\b\in \R^{k-1}, ~\c\in \R^{n-k} \text{ and }a>0, \]
be $\{k,k\}$-irreducible. There exists $\u\in \R^n_+$ such that
 $\u^TA\u=0$  and $u_k=1$, i.e.,    $ \u^T=(\v^T~ 1~\w^T)$ for some $\v\in \R^{k-1}_+$ and $\w\in \R^{n-k}_{+}$.

 Let \[A_1=\left(\begin{array}{cc}
                                                    B_1 & \b \\
                                                    \b^T & p \\
                                                  \end{array}
                                                \right), \] where $p$ is such that $(\v^T~1)A_1\left(\begin{array}{c}
         \v \\
        1 \\
        \end{array}\right)=0$.
By the copositivity of $A[1,\dots, k]$ such $p$ exists, and $p\le a$.
Let  $q=a-p$ and \[\quad  A_2=\left(\begin{array}{cc} q & \c^T \\
                                                    \c & B_2 \\
                                                  \end{array}
                                                \right). \]
Since $A=(A_1\oplus 0)+(0\oplus A_2)$,  \[0=\u^TA\u =(\v^T~1)A_1 \left(\begin{array}{c}
         \v \\
        1 \\
        \end{array}\right)+    (1~\w^T)A_2\left(\begin{array}{c}
         1 \\
        \w \\
        \end{array}\right),\]
and thus also  $(1~\w^T)A_2\left(\begin{array}{c}
         1 \\
        \w \\
        \end{array}\right)=0$.

We now show that the matrix $A_1$ is copositive.  If $\x\in \R^k_+$ has the form $\x^T=(\y^T~ 0)$, then $\x^TA_1\x=\y^TB_1\y\ge 0$. Otherwise $\x^T=(\y^T~ \gamma)$, $\gamma>0$,
and we may assume $\gamma=1$. Then
for $\z=(\y^T ~1 ~\w^T)^T$ we have
\[0\le \z^TA\z=(\y^T ~1)A_1\left(\begin{array}{c}
         \y \\
        1 \\
        \end{array}\right)+(1~\w^T)A_2\left(\begin{array}{c}
                                                           1 \\
                                                           \w \\
                                                         \end{array}
                                                       \right)=\x^TA_1\x\]
By the copositivity of $A_1$, $p\ge 0$, and thus also $0\le q\le a$, and
by a similar argument $A_2$ is also copositive.
\end{proof}

\section{Initial results on SPN graphs}
In this section we present some basic results on SPN graphs.

\begin{lemma}\label{lem:SPN by components}
A graph $G$ is SPN if and only if every connected component of $G$ is SPN.
\end{lemma}

\begin{proof}
This follows immediately from the fact that $A_1\oplus A_2$ is copositive if and only if $A_1$ and $A_2$ are both copositive,
and $A_1\oplus A_2$ is SPN if and only if $A_1$ and $A_2$ are both SPN.
\end{proof}

\begin{lemma}\label{lem:subgraph}
If $G$ is an SPN graph, then every subgraph of $G$ is an SPN graph.
\end{lemma}

\begin{proof}
Let $H$ be any subgraph of $G$. Let $A$ be a copositive matrix with $G(A)=H$. Without loss of generality we may assume  that $V(H)=\{1,\dots, k\}$ for some $k\le n$,
and that the positive diagonal entries of $A$ are equal to $1$.
If $V(G)\setminus V(H)\cob{\ne \emptyset}$, extend $A$ to an $n\times n$ matrix $\widehat{A}=(\widehat{a}_{ij})\in \cop_n$ by setting $\widehat{A}=A\oplus I$.
For every $\varepsilon>0$ let $A_\varepsilon$ be defined by
\[\left(A_\varepsilon\right)_{ij}=\begin{cases}
\varepsilon \qquad & \text{if   $ij\in E(G)\setminus E(H)$}\\
 \widehat{a}_{ij} &\text{otherwise}
 \end{cases}\, .\]
Then $A_\varepsilon\ge \widehat{A}$ is copositive, and $G(A_\varepsilon)=G$. Since $G$ is SPN, the matrix $A\varepsilon$ is SPN for every $\varepsilon>0$.
That is, for every $\varepsilon>0$ there exists a positive semidefinite $P_\varepsilon$ such that $A_\varepsilon\ge P_\varepsilon$ and
$\diag P_\varepsilon =\diag A_\varepsilon$. (Note that whenever a diagonal entry of $A_\varepsilon$ is zero, the corresponding entry in $P_\varepsilon$, and any entry in
the row and column which intersect at it, is zero.)
The sequence $ P_{1/k}$, $k\ge 1$, is bounded ($-J\le P_{1/k}\le J$ for every $k$), and thus has a converging subsequence $P_{1/k_\ell}$, which converges to a
positive semidefinite $P$. Since $A_{1/k_\ell}\ge P_{1/k_\ell}$ for every $\ell$, and $\lim_{\ell\to \infty}A_{1/k_\ell}=\widehat{A}$,
we get that $\widehat{A}\ge P$. That is, $\widehat A$ is SPN, and thus its principal submatrix $A$ is SPN.
\end{proof}

\begin{lemma}\label{lem:cut vertex}
Let  $G=G_1\cup G_2$, where $G_1$ and $G_2$ are SPN graphs and $G_1\cap G_2$ is a single vertex. Then $G$ is an SPN graph if and only if $G_1$ and $G_2$ are SPN graphs.
\end{lemma}

\begin{proof}
The `only if' is clear from Lemma \ref{lem:subgraph}. The reverse follows from Lemma \ref{lem:semidirectsum}.
\end{proof}

By induction, we get the following corollary (recall that a \emph{block} of $G$ is a subgraph that has no cut vertex, and is maximal with respect to this property).

\begin{corollary}\label{cor:SPNbyblocks}
A graph $G$ is SPN if and only if every block of $G$ is SPN.
\end{corollary}

\section{The case $n\le 4$ and implications}
The proof by Diananda \cite{Diananda1962} that $\cop_n=\spn_n$ for $n\le 4$ is one of the earliest results in the theory
of copositive matrices. Diananda's original proof was in terms of quadratic forms, and in the case of $n=4$ it
required discussion of quite a few subcases. In \cite{LiFeng1993} another proof was presented for the case $n=4$, this time in terms of matrices,
and it required a discussion of eight different subcases according to the signs of the off-diagonal elements.
We demonstrate here how  using graphs (together with characterization of $\{i,j\}$-irreducibility
in Lemma \ref{lem:ij irred}) generates a proof  from general principles of Diananda's result, which is relatively short.

\begin{theorem}\label{thm:n<5}
If $A\in \cop_n$, $n\le 4$, then $A\in \spn_n$.
\end{theorem}

\begin{proof}
For $n=1$ the claim is trivial, and it is easy to see from Lemma \ref{lem:offdiag} that
$\cop_2=\nn_2\cup\psd_2=\spn_2$. For $n=3$, let $A\in \cop_3$. By Corollary \ref{cor:minimlally non SPN}, it suffices to consider such $A$ with $G_-(A)$ connected. Then $G_-(A)$ has a vertex of degree $2$, i.e., $A$ has a nonpositive row. The claim then follows from the case $n=2$ by Lemma  \ref{lem:+row-row}(b).

For $n=4$: Let
$A\in \cop_4$. It suffices to consider $A$ that is  $\widetilde{\mathcal{N}}$-irreducible and such that $G_-(A)$ is connected.
If $A$ has a row
in which all the off-diagonal entries are nonpositive ($G_+(A)$ has an isolated vertex), then it is SPN by Lemma \ref{lem:+row-row}(a) and the case $n=3$.
If $G_-(A)$ contains a triangle, then $A$ has a $3\times 3$ positive semidefinite submatrix (Lemma \ref{lem:Zmat}), and is
thus itself SPN (Lemma \ref{lem:n-1psd}).

Now suppose the connected  $G_-(A)$  is triangle free  and $G_+(A)$ has no isolated vertex. Then $G_-(A)$ is either
 a path of length $3$, or a $4$-cycle. In any case, $G_-(A)$ contains a path of length $3$, say $1-2-3-4$, such that  $a_{13}>0$ and $a_{34}>0$.
That is, we may assume that the sign pattern of $A$ is
\[\left(\begin{array}{cccc}
            + & - & + & * \\
            - & + & - & + \\
            + & - & + & - \\
            * & + & - & +
          \end{array}
\right),\]
where $*$ denotes an entry which may be either $0$, or positive, or negative.
As $A$ is $\{1,3\}$-irreducible, there exists  $\u\in \mathcal{V}^{A}$ such that $u_1+u_3>0$ and $(A\u)_1=(A\u)_3=0$. From $u_1+u_3>0$ and $(A\u)_3=0$ we deduce
that
$u_2+u_4>0$. In particular, $(A\u)_i=0$ for either $i=2$ or $i=4$. Thus $|\supp \u|\ge 2$ and $|\{i\, |\, (A\u)_i=0\}|\ge 3$, implying that $A$ is SPN by Lemma \ref{lem:suppu>=n-2+}.
\end{proof}

\begin{theorem}\label{thm:V(G)<5}
Every graph on at most $4$ vertices is an SPN graph.
\end{theorem}

Combining Theorem \ref{thm:V(G)<5} with Corollary \ref{cor:SPNbyblocks} we get the following.

\begin{corollary}\label{cor:4blocks}
A graph $G$ in which every block has at most $4$ vertices is SPN.
\end{corollary}

In particular, any acyclic graph (a forest) is an SPN graph, but more can be said in this case (and used to identify
SPN/copositive matrices whose graphs are acyclic).
In the next theorem we use the following definition:
For $A\in \sym$, $N(A)$ is the symmetric matrix of the same order defined by:
\[(N(A))_{ij}=\begin{cases} a_{ij}\quad & \text{if   $i=j$ or $a_{ij}<0$}\\
0 &\text{if  $i\ne j$ and $a_{ij}\ge 0$}
 \end{cases}\, .\]
 Note that $G(N(A))=G_-(A)$.

\begin{theorem}\label{thm:minA,0}
If $G$ is acyclic, then for a matrix $A\in \sym$ with $G(A)=G$ the following are equivalent:
\begin{enumerate}
\item[\rm{(a)}] $A\in \cop$
\item[\rm{(b)}] $A\in\spn$
\item[\rm{(c)}] $N(A)\in \psd$.
\end{enumerate}
\end{theorem}

\begin{proof}
Obviously, (b) implies (a)  and, since $A\ge N(A)$, (c) implies (b). Finally, (a)$\Rightarrow$(c) follows
from  Corollary \ref{cor:-induced}, since in the case that $G(A)=G$ is acyclic,  $G_-(A)$ is a disjoint union of induced subgraphs of $G(A)$.
\end{proof}

\begin{corollary}\label{cor:K1p}
The matrix $A=\left(
        \begin{array}{cc}
          I & \v \\
          \v^T & 1 \\
        \end{array}
      \right)$
is SPN if and only if $||\v_-||\le 1$.
\end{corollary}

\begin{proof}  By Theorem \ref{thm:minA,0}, the matrix $A$ is SPN if and only if $N(A)$ is positive semidefinite.
But
$N(A)=\left(
        \begin{array}{cc}
          I & -\v_- \\
          -\v_-^T & 1 \\
        \end{array}
      \right)$  is positive semidefinite if and only if $I-\v_-\v_-^T$ is positive semidefinite, i.e., if and only if $||\v_-||\le 1$.
\end{proof}

\section{Every $T_n$ is SPN}
Let $T_n$ denote the graph on $n$ vertices consisting of $n-2$ triangles with a common base. In this section we prove that every $T_n$ is an SPN graph.

\begin{center}
\begin{tikzpicture}[scale=0.8]
\draw[thick] (7.5,0)--(10.5, 0)--(9,1)--(7.5,0)--(9,2)--(10.5,0)--(9,3)--cycle;

\draw[fill](7.5,0) circle[radius=0.1];
\draw[fill](10.5,0) circle[radius=0.1];
\draw[fill](9,1) circle[radius=0.1];
\draw[fill](9,2) circle[radius=0.1];
\draw[fill](9,3) circle[radius=0.1];

\node[align=center, below] at  (9,-0.7){\fig{fig:T5} : $T_5$ };
\end{tikzpicture}
\end{center}

In the next three results we consider matrices whose graph is $T_n$ of specific sign patterns.
In Lemma \ref{lem:pre-Tn} $G_-(A)$ is the star $K_{1,n-1}$.

\begin{lemma}\label{lem:pre-Tn}
Let
\[A=\left( \begin{array}{ccc}
                    I & -\u & \b \\
                    -\u^T & 1 & -r \\
                    \b^T & -r & 1 \\
                  \end{array}
                \right) \in \cop_n,\]
                where $\u, \b\in \R^{n-2}_{++}$,   and $r\in \R_{++}$. Then $A\in \spn_n$. Moreover, let $0\le s\le 1$ \cob{be}\ignore{is} such that
                the matrix
$A_s$, obtained from $A$ by replacing the $n,n$ entry by $s$, is $\{n,n\}$-irreducible. Then there exists
 $\d \le \b$ such that
\[B=\left( \begin{array}{ccc}
                    I & -\u & \d \\
                    -\u^T & 1 & -r \\
                    \d^T & -r & s \\
                  \end{array}
                \right) \]
                is positive semidefinite, and \begin{equation}\left(\begin{array}{cc} \d^T &-r\end{array}\right)\left(\begin{array}{cc}
                    I & -\u \\
                    -\u^T & 1\\
                  \end{array}\right)^\dag \left(\begin{array}{c} \d\\-r\end{array}\right)=s .\label{eq:yTMdaggMy}\end{equation}
\end{lemma}

\begin{proof}
By Corollary \ref{cor:K1p} $A(n)$ is positive semidefinite, and thus $A\in \spn$ by Lemma \ref{lem:n-1psd}. By the proof of the latter,
 \[A\ge B=\left(\begin{array}{cc}
           M & -M\x \\
           -\x^TM & s
         \end{array}
 \right) \in \psd,\]
where $M=\left(\begin{array}{cc}
                    I & -\u \\
                    -\u^T & 1\\
                  \end{array}\right)$ and $\w=\left(\begin{array}{c} \x \\1\end{array}\right)\in \mathcal{V}^{A_s}$.
Since $w_n=1$ and $(A_s\w)_n=0$, the zero pattern of the last row in $A_s$ implies that $w_{n-1}>0$.
Thus $(A_s\w)_{n-1}=0$, which implies that $(-M\x)_{n-1}=-r$ (see final comment in the proof of Lemma \ref{lem:n-1psd}).
Let $-M\x=\left(\begin{array}{c} \d\\-r\end{array}\right)$.
By the same proof, we have $\x^TM\x=s$, thus \eqref{eq:yTMdaggMy} holds.
\end{proof}

The next lemma considers a  matrix $A$ whose graph is $T_n$ for which  $G_-(A)$ is a double star (two stars whose centers
are joined by an edge).

\begin{lemma}\label{lem:part T_n}
Let
\[A=\left(   \begin{array}{cccc}
              I & -\u & \b & 0 \\
              -\u^T & 1 & -r & \c^T \\
              \b^T & -r & 1 & -\v^T \\
              0 & \c & -\v & I \\
            \end{array}
          \right)
\in \sym_n,\]
where $n=k+m+2$, $\u, \b\in \R^k_{++}$, $\v,\c\in \R^m_{++}$,  and $r>0$. If $A[1,\dots, k+2]\in \cop_{k+2}$,
$A[k+1,\dots, n]\in \cop_{m+2}$ and either $\|\cob{\u}\|=1$ or $\|\cob{\v}\|=1$. Then $A\in \spn_n$.
\end{lemma}

\begin{proof}
Without loss of generality suppose $\|\u\|=1$.
For convenience we denote \[M= \left(\begin{array}{cc}
                    I & -\u \\
                    -\u^T & 1\\
                  \end{array}\right)\  \text{ and }\   N=\left(\begin{array}{cc}
                    1 & -\v^T \\
                    -\v & I\\
                  \end{array}\right)\]
As $A$ is copositive, $M$ and $N$ are positive semidefinite, and $M$ is singular since $\|\u\|=1$.

Apply Lemma \ref{lem:pre-Tn} \cob{to} $A_1=A[1,\dots, k+2]$ to obtain \[A_1\ge P_1=\left( \begin{array}{ccc}
                    I & -\u & \d \\
                    -\u^T & 1 & -r \\
                    \d^T & -r & s \\
                  \end{array}
                \right)\in \psd_{k+2}, \]
with $\d\le \b$ and $s=\left(\begin{array}{cc} \d^T &-r\end{array}\right)M^\dag \left(\begin{array}{c} \d \\-r\end{array}\right)$.
Then $P_1/I=\left(
               \begin{array}{cc}
                 0 &  \u^T\d-r \\
                 \d^T\u-r & s-r^2 \\
               \end{array}
             \right)$ is positive semidefinite, implying that $\u^T\d=r$.
Thus  \[P_1'=\left( \begin{array}{ccc}
                    I & -\u & \d \\
                    -\u^T & 1 & -r \\
                    \d^T & -r & r^2 \\
                  \end{array}
                \right)  \]
is also positive semidefinite (since $P_1'/I=0$). By the minimality of  $s$,  this implies $s=r^2$ and $P_1'=P_1$.

Let $0\le t\le 1$ be minimal such that
\[\left(
          \begin{array}{ccc}
            t & -r & \c^T \\
            -r & 1 & -\v^T \\
            \c & \v & I \\
          \end{array}
        \right)\]
is copositive. Applying Lemma \ref{lem:pre-Tn}  to $A_2=A[k+1,\dots, n]$ we get that
\[A_2\ge P_2=\left(
          \begin{array}{ccc}
            t & -r & \e^T \\
            -r & 1 & -\v^T \\
            \e & \v & I \\
          \end{array}
        \right) \in \psd_{m+2},\]
        where $\e\le \c$. By the proofs of Lemmas \ref{lem:n-1psd} and \ref{lem:pre-Tn},
        \[\left(\begin{array}{c}
        \d\\
          -r \\
          \end{array}\right)=-M\x \quad \text{ and }\quad \left(\begin{array}{c}
        -r\\
          \e \\
          \end{array}\right)=-N\y,\]
          for some $\x\in \R^{k+1}_+ $ and $\y\in \R^{m+1}_+$. In particular, $r^2=s=(-M\x)^TM^\dag(-M\x)=\x^TM\x$.

Denote $E=\left(\begin{array}{cc}
            \b & 0 \\
            -r &\c^T
          \end{array}
\right)$, and let $F=-\frac{1}{\cob{r}}(M\x)(N\y)^T=\left(\begin{array}{cc}
                               \d & -\frac{1}{r}\d\e^T \\
                               -r &   \e^T
                             \end{array}
 \right)$.
 Then
 \[P=\left(\begin{array}{cc}
            M & F \\
            F^T &N
          \end{array}
\right) \le \left(\begin{array}{cc}
            M & E \\
            E^T &N
          \end{array}
\right)= A.\]
We complete the proof by showing that $P$ is positive semidefinite.  Since $M$ is positive semidefinite and the column space of $F$, spanned by $M\x$, is contained in the column space of $M$, it suffices to prove
that $P/M=N-F^T M^\dag F$ is positive semidefinite.

We have
\begin{eqnarray*}P/M&=&N-F^T M^\dag F\\&=& N-\frac{1}{r^2}    (N\y) (M\x)^T  M^\dag(M\x)(N\y)^T
\\&=&N-\frac{1}{r^2} N\y\x^T M\x\y^TN
\\&=&\sqrt{N}(I-(\sqrt{N}\y)(\sqrt{N}\y)^T )\sqrt{N} \end{eqnarray*}
since $\x^T M\x=s=r^2$. But
\[\|\sqrt{N}\y\|^2=\trace( (\sqrt{N}\y)(\sqrt{N}\y)^T)=(\sqrt{N}\y)^T(\sqrt{N}\y)=\y^TN\y=t\le 1,\]
and thus the matrix  $I-(\sqrt{N}\y)(\sqrt{N}\y)^T$   is positive semidefinite, implying that $P/M$ is positive semidefinite.
\end{proof}

In the next corollary we show that the assumption in the previous lemma that $\|\u\|=1$ or $\|\v\|=1$ can be dropped.

\begin{corollary}\label{cor:almostTn} Let
\[A=\left(
            \begin{array}{cccc}
              I & -\u & \b & 0 \\
              -\u^T & 1 & -r & \c^T \\
              \b^T & -r & 1 & -\v^T \\
              0 & \c & -\v & I \\
            \end{array}
          \right)
 \in \sym_n,\]
where $\u, \b\in \R^k_{++}$, $\v,\c\in \R^m_{++}$, $n=k+m+2$, and $r>0$. If $A[1,\dots, k+2]\in \cop_{k+2}$,
$A[k+1,\dots, n]\in \cop_{m+2}$, then $A\in \spn_n$.
\end{corollary}

\begin{proof}
If $\|\u\|=1$ or $\|\v\|=1$, this is known from Lemma \ref{lem:part T_n}. So suppose $\|\u\|<1$. Let $u_0=\sqrt{1-\|\u\|^2}$, and extend $A$ to $\widehat{A}\in \sym_{n+1}$ (labeling
its rows by $0,1,\dots, n$) as follows.
\[\widehat{A}=\left(
            \begin{array}{ccccc}
            1&\0^T&-u_0 & 1& \0^T\\
             \0& I & -\u & \b & 0 \\
             -u_0& -\u^T & 1 & -r & \c^T \\
             1& \b^T & -r & 1 & -\v^T \\
              \0&0 & \c & -\v & I \\
            \end{array}
          \right)
.\]
Every principal submatrix of $\widehat{A}[0,\dots, k+2]$ with off-diagonal elements smaller than $1$ is
either a submatrix of $\widehat{A}[0,\dots, k+1]$, which is positive semidefinite, or a submatrix of $\widehat{A}[1,\dots,k+2]=A[1,\dots, k+2]$, which
is copositive by the assumptions. Thus  $\widehat{A}[0,\dots, k+2]\in \cop_{k+3}$ and
$\widehat{A}[k+1,\dots, n]=A[k+1,\dots, n]\in \cop_{m+2}$. The vector $\left(\begin{array}{c} u_0\\ \u\end{array}\right)$ has norm $1$, and thus by Lemma \ref{lem:part T_n} the matrix
$\widehat{A}$ is SPN. Hence its principal submatrix $A$ is SPN.
\end{proof}

\begin{theorem}\label{thm:Tn}
For every $n\ge 3$ the graph $T_n$ is SPN.
\end{theorem}

\begin{proof}
Obviously $T_3$ (the triangle) and $T_4$ (the diamond) are SPN. We proceed by induction on $n$. Suppose $n\ge 5$ and $A\in \cop_n$ has $G(A)=T_n$. We may assume that $\diag A =\1$.

For every vertex $i$ such that the degree of $i$ in $G(A)$ is $2$ we have that $G(A(i))=T_{n-1}$ and $G(A/A[i])$ is a subgraph of $T_{n-1}$. Thus if in the signed graph  $\mathcal{G}(A)$ such $i$
is incident with two  positive edges, or with two  negative edges, then $A$ is SPN by the induction hypothesis.
We therefore consider the case that in the signed graph each vertex of degree $2$ is incident with one positive edge and one negative edge. If the
base edge is positive, then $G_-(A)$ is not connected, and $A$ is SPN by  Corollaries  \ref{cor:-induced} and \ref{cor:K1p}.
If the base edge is negative in $\mathcal{G}(A)$, then $A$ satisfies the conditions in Corllary \ref{cor:almostTn}, and is therefore SPN.
\end{proof}

\section{The case $n=5$}
In this section we fully characterize the SPN graphs on $5$ vertices.
There exist graphs on $5$ vertices, which are not SPN graphs. Our first example of a non-SPN graph is the fan graph $F_5$,
which consists of a path of length $3$ and an additional vertex connected to all four vertices of the path.

\begin{lemma}\label{lem:F5}
The fan graph $F_5$ is not an SPN graph.
\end{lemma}

\begin{center}
\begin{tikzpicture}[scale=0.8]
\draw[thick] (0,1)--(1, 0)--(2.5,0)--(3.5,1)--(1.75,2.5)--cycle;
\draw[thick]  (1,0)--(1.75,2.5);
\draw[thick]  (2.5,0)--(1.75,2.5);
\draw[thick] (0,1)--(1, 0);
\draw[thick] (2.5,0)--(3.5,1);

\draw[fill](0,1) circle[radius=0.1];
\draw[fill](1,0) circle[radius=0.1];
\draw[fill](2.5,0) circle[radius=0.1];
\draw[fill](3.5,1) circle[radius=0.1];
\draw[fill](1.75,2.5) circle[radius=0.1];

\node[left] at (-0.1,1) {1};
\node[below] at (1,-0.1) {2};
\node[below] at (2.5,-0.1) {4};
\node[right] at (3.6,1) {5};
\node[above] at (1.75,2.6) {3};

\node[align=center, below] at  (1.75,-0.7){\fig{fig:F5} : $F_5$};

\end{tikzpicture}
\end{center}

\begin{proof}
The matrix \[A=\left(\begin{array}{rrrrr}
               1 & -1 & 1 & 0 & 0 \\
               -1 & 1 & -1 & 1 & 0 \\
               1 &-1 &1 & -1 & 1\\
               0 & 1 & -1 & 1 & -1 \\
               0& 0 &  1 &- 1 & 1
             \end{array}
\right) \]
is a  $\{0,\pm 1\}$-matrix, $G_{-1}(A)$ is triangle free and $a_{ij}=1$ whenever $d_{G_{-1}}(i,j)=2$. Therefore  $A$ is copositive
by Lemma \ref{lem:01-1cop}. By Lemma  \ref{lem:G-1SPN} it is not SPN, as $d_{G_{-1}(A)}(1,5)=4$ is even, but $a_{15}\ne 1$. The graph $G(A)$ is the fan graph $F_5$ shown in
Fig. \ref{fig:F5}.
\end{proof}

The next lemma gives an example of a 2-connected SPN graph on $5$ vertices.
Let $DR_5$ denote the graph obtained from $K_4$ by replacing one edge by an independent path of length $2$.

\begin{center}
\begin{tikzpicture}[scale=0.8]
\draw[thick] (0,0)--(4, 0)--(2,3)--(2,1.2)--cycle;
\draw[thick] (0,0)--(2,3);
\draw[thick] (2,1.2)--(4,0);

\draw[fill](0,0) circle[radius=0.1];
\draw[fill](4,0) circle[radius=0.1];
\draw[fill](2,1.2) circle[radius=0.1];
\draw[fill](2,3) circle[radius=0.1];
\draw[fill](2,0) circle[radius=0.1];

\node[align=center, below] at  (2,-0.7){\fig{fig:DR5}: $DR_5$};
\end{tikzpicture}
\end{center}

\begin{lemma}\label{lem:DR5}
The graph $DR_5$ is an SPN graph.
\end{lemma}

\begin{proof}
Let $A\in \cop_5$ have $G(A)=DR_5$. By Corollary \ref{cor:minimlally non SPN} it suffices to consider the case that $G_-(A)$ is connected, since every proper subgraph of $DR_5$ is SPN. As
every $4\times 4$ copositive matrix is SPN, we may assume that $A$ has no row with all off-diagonal entries nonpositive, i.e., $G_+(A)$ has no isolated vertices.
Combining these assumptions and the fact that there are $7$ edges in $DR_5$, we find that there are exactly $4$ negative edges
and $3$ positive edges in $\mathcal{G}(A)$.
Up to isomorphism, $\mathcal{G}(A)$ has to be one of the two signed graphs shown in Fig. \ref{fig:signedDR5}.
In both cases, we label the vertices so that $G_-(A)$ consists of
the path $1-2-3-4-5$.

\begin{center}
\begin{tikzpicture}[scale=0.8]
\draw[thick, densely dashed] (2,0)--(0, 0)--(2,3)--(4,0)--(2,1.2);
\draw[thick] (0,0)--(2,1.2)--(2,3);
\draw[thick] (2,0)--(4,0);

\draw[fill](0,0) circle[radius=0.1];
\draw[fill](4,0) circle[radius=0.1];
\draw[fill](2,1.2) circle[radius=0.1];
\draw[fill](2,3) circle[radius=0.1];
\draw[fill](2,0) circle[radius=0.1];

\node[left] at (0,0) {2};
\node[below] at (2,0) {1};
\node[below] at (2,1.2) {5};
\node[right] at (4,0) {4};
\node[above] at (2,3) {3};

\node[align=center, below] at  (2,-0.7){Case $1$};

\draw[thick, densely dashed] (10,0)--(8, 0)--(10,3)--(10,1.2)--(12,0);
\draw[thick] (8,0)--(10,1.2);
\draw[thick] (10,3)--(12,0)--(10,0);

\draw[fill](8,0) circle[radius=0.1];
\draw[fill](12,0) circle[radius=0.1];
\draw[fill](10,1.2) circle[radius=0.1];
\draw[fill](10,3) circle[radius=0.1];
\draw[fill](10,0) circle[radius=0.1];

\node[left] at (8,0) {2};
\node[below] at (10,0) {1};
\node[below] at (10,1.2) {4};
\node[right] at (12,0) {5};
\node[above] at (10,3) {3};

\node[align=center, below] at  (10,-0.7){Case $2$};

\node[align=center, below] at  (6,-1.4){\fig{fig:signedDR5}: signed $DR_5$};
\end{tikzpicture}
\end{center}

We consider each of these cases:

{\bf Case $1$}: In this case, the sign pattern of $A$ is
\[\left(\begin{array}{ccccc}
    + & - & 0 & + & 0 \\
    - & + & - & 0 & + \\
    0 & - & + & - & + \\
    + & 0 & - & + & - \\
    0 & + & + & - & +
  \end{array}\right)
\]
Let $B=A-\delta E_{14}$ be a $\{1,4\}$-irreducible copositive matrix. If $b_{14}\le 0$, then the first row of $B$ has all off-diagonal entries
nonpositive and $B$ is SPN. Thus we consider the case that $b_{14}>0$, i.e.,  $\mathcal{G}(B)=\mathcal{G}(A)$. Then there exists $\u\in \mathcal{V}^B$
such that $u_1+u_4>0$ and $(B\u)_1=(B\u)_4=0$. From $(B\u)_1=0$ and $u_1+u_4>0$ we deduce that $u_2>0$, and from $(B\u)_4=0$ and $u_1+u_4>0$ we deduce that
$u_3>0$. Thus $|\supp \u|\ge 3$ and $|\{i\, |\, (B\u)_i=0\}|\ge 4$. By Lemma \ref{lem:suppu>=n-2+},  $B$ is SPN  and therefore so is $A$.

{\bf Case $2$}: In this case, the sign pattern of $A$ is
\[\left(\begin{array}{ccccc}
    + & - & 0 & 0 & + \\
    - & + & - & + & 0 \\
    0 & - & + & - & + \\
    0 & + & - & + & - \\
    + & 0 & + & - & +
  \end{array}\right)
\]
This time let $B=A-\delta E_{15}$ be a $\{1,5\}$-irreducible copositive matrix.  It suffices to consider the case that $\mathcal{G}(B)=\mathcal{G}(A)$.
Then there exists $\u\in \mathcal{V}^B$
such that $u_1+u_5>0$ and $(B\u)_1=(B\u)_5=0$. As above, this implies that $u_2>0$ and $u_4>0$. Thus $(B\u)_2=(B\u)_4=0$, so $|\supp \u|\ge 3$ and $|\{i\, |\, (B\u)_i=0\}|\ge
4$.
By Lemma \ref{lem:suppu>=n-2+} $B$ is SPN, and therefore so is $A$.
\end{proof}

We can now characterize the SPN graphs on $5$ vertices. In the proof use the following standard notations: $P_n$  denotes a path on $n$ vertices,  $K_n$ is the complete graph on $n$ vertices, $K_{m,k}$ denotes the complete bipartite graph with independent sets of vertices of sizes $m$ and $k$, and $mK_n$ are $m$ disjoint copies of $K_n$.

\begin{theorem}\label{thm:spn on 5}
A graph on $5$ vertices is an SPN graph if and only if it does not contain the fan $F_5$.
\end{theorem}

\begin{proof}  By Lemma \ref{lem:F5}, $F_5$ is not an SPN graph.
This, together with Lemma \ref{lem:subgraph}, proves the `only if' part.
For the `if' part, let $G$ be a graph on $5$ vertices that does not have an $F_5$ as a subgraph.
The complement of $G$ has
at least $3$ edges, since otherwise the complement of $G$ would be a subgraph of $P_4\cup K_1$, which is the complement of $F_5$, and $G$ would contain $F_5$. Thus $G$ has at
most $7$ edges. Let $A\in \cop_5$ have $G(A)=G$. If $G_-(A)$ is not connected, $A$ is \cob{SPN} by Lemma \ref{lem:A-} and the case $n\le 4$. We therefore consider the
case that $G_-(A)$ is connected. In particular, $G_-(A)$ has at least $4$ edges. If $G_+(A)$ has at most $2$
edges, then at most $4$ rows have positive off-diagonal elements. In this case, there is a row $i$ with all off-diagonal entries nonpositive, the matrix $A/A[i]\in \cop_4$ is
SPN by the case $n=4$, and thus $A$ is SPN.  We assume therefore that $G_+(A)$ has at least three edges. Combining all the assumptions  on the number of edges in $G$, $G_-(A)$ and $G_+(A)$, the proof
is reduced to
the case that $G\ne F_5$  has exactly $7$ edges. Since the complement of $G$
has $3$ edges, and it is not $P_4\cup K_1$, it has to be  either $K_3\cup 2K_1$, or $P_2\cup P_3$. In the first case $G=T_5$ and
in the second $G=DR_5$, thus in both cases $A$ is SPN.
\end{proof}

\section{Graph transformations}
In this section we check whether being SPN or being non-SPN is a graph property that is preserved by certain  graph transformations.
First, consider subdivision. Recall that a graph $\widehat{G}$ is a \emph{subdivision} of $G$ if it is obtained
from $G$ by replacing edges of $G$ by independent paths. Any edge replaced by a path of length at least two is said to be
\emph{subdivided}. We will say the edge is \emph{subdivided $k$ times} if it is replaced by a path of length $k+1$.
Any graph $G$ is considered to be its own subdivision (where no edge
is actually subdivided).

It turns out that it is more useful to consider the effect of subdivision on a signed graph. We say that a path is \emph{negative (positive)}
if all its edges are negative (positive).

\begin{lemma}\label{lem:subdivide-}
Let $\mathcal{G}$ be a signed graph, and let $\widehat{\mathcal{G}}$ be the signed graph obtained from $\mathcal{G}$
by replacing a negative edge by a negative path of length $2$. Then  $\mathcal{G}$ is SPN if and only if $\widehat{\mathcal{G}}$ is.
\end{lemma}

\begin{proof}
Suppose $\mathcal{G}$ is not SPN.
Let $A\in \cop_n\setminus\spn_n$ have signed graph $\mathcal{G}$, and let $ij$ be the negative edge that is being subdivided. For convenience we may  assume that $i=n-1$ and $j=n$ and that $a_{ n-1,n}=-1$. Let
$\widehat{A}=(A\oplus 0)+ (0\oplus E)\in \sym_{n+1}$, where $E=\left(
                                                     \begin{array}{ccc}
                                                       1 & 1 & -1 \\
                                                       1 & 1 & -1 \\
                                                       -1 & -1 & 1 \\
                                                     \end{array}
                                                   \right)$.
Then $\mathcal{G}(\widehat{A})=\widehat{\mathcal{G}}$, $\widehat A$ is copositive as a sum of a copositive
matrix and a positive semidefinite one. The matrix $\widehat A$ has only nonpositive off-diagonal entries in its last row and column, and it is
not SPN since  $A/A[n+1]=A$ is not SPN.

Now suppose $\mathcal{G}$ is SPN.
Let $A\in \cop_n$ have $\mathcal{G}(A)=\widehat{\mathcal{G}}$, and let $i-k-j$ be a negative path of length $2$ obtained by subdividing the edge $ij$ in $\mathcal{G}$.
Then $\mathcal{G}(A/A[k])=\mathcal{G}$, thus $A/A[k]$ is SPN, and by Lemma \ref{lem:+row-row} so is $A$.
\end{proof}

By repeated use of the lemma we get the following corollary.

\begin{corollary}\label{cor:rep subdivision}
Let $\mathcal{G}$ be a signed graph. Then any signed graph obtained from $\mathcal{G}$ by replacing
a negative edge by an negative independent path is SPN if and only if $\mathcal{G}$ is SPN.
\end{corollary}

In the non-signed case, subdivision of an SPN graph can yield a non-SPN graph, unless the subdivided edge lies on an independent path $P$ of length at least 3.

\begin{lemma}\label{lem:P4}
Let a graph $G$  contain  an independent path $P$ of length $\ell\ge 3$. Then the graph $\widehat G$ obtained
from $G$ by subdividing one of the edges of $P$ is an SPN graph if and only if $G$ is an SPN graph.
\end{lemma}

\begin{proof}
Let $\widehat P$ be the path in $\widehat G$ obtained from subdividing an edge of $P$.
Note that for any edge $e$ of the path $\widehat P$, $\widehat{G} -e$ is SPN  if and only if $G$ is,
by Corollary \ref{cor:SPNbyblocks} (and the fact that any single-edge block is SPN).

Suppose first that $G$ is an SPN graph, and let
 $A\in \cop$ have graph $\widehat G$. For any internal vertex $i$ of $\widehat P$,  $G(A(i))$ is
SPN, and $G(A/A[i])=G$ is also SPN. Thus if both edges incident with such $i$ in  $\mathcal{G}(A)$
are positive or both are negative, then $A\in \spn$. Otherwise, the edges of
$\widehat P$ have alternating signs in $\widehat{\mathcal{G}}(A)$, and since there are at least three inner
vertices, there exists a negative edge $e$ on the signed path between two positive edges. In that case,
$G_-(A)$ is not connected. The subgraphs of $\widehat{G}$ induced by the vertex sets of the connected
components of $G_-(A)$ are all either single edges or subgraphs $\widehat{G} -e$, so they are all SPN.
By Lemma \ref{lem:A-}, $A$ is SPN.

If $G$ is not SPN, there is a sign assignment to its edges that yields a non-SPN signed graph $\mathcal{G}$.
If one of the edges of $P$ is negative in $\mathcal{G}$, it can be subdivided
to obtain a non-SPN graph whose underlying graph is $\widehat G$. If $P$ is a positive path in $\mathcal{G}$,
then the signed graph $\mathcal{G}'$ obtained from   $\mathcal{G}$ by omitting the edges of $P$ is not
SPN by Lemma \ref{lem:M+-rows}. By the same lemma, the signed graph $\mathcal{G}$ obtained by adding to $\mathcal{G}'$ the path $\widehat{P}$ with
all positive edges is also not SPN, and therefore its underlying graph $\widehat{G}$ is not SPN.
\end{proof}

Next we consider another transformation on  signed graphs.
A \emph{paw graph}  has two blocks, a triangle and an edge. We call the following transformation on
signed graph a \emph{$\Lambda$-paw transformation} on $x-y-z$: If $x-y-z$ is a negative path in $\mathcal{G}$, where
the end vertices $x$ and $z$ may or may not be adjacent (by an edge of any sign), \cob{erase the negative path $x-y-z$,} add a new vertex
$w$ to $\mathcal{G}$, connected to each of $x$, $y$ and $z$ by a negative edge, and add (or replace the existing edge by)
a positive edge $xz$. Fig. \ref{fig:Lambdapaw} describes this transformation (with a dashed line denoting a negative edge,
a solid line a positive edge, and a double dotted line denoting a possible, but not mandatory, edge of either sign).

\begin{center}{\begin{tikzpicture}[scale=0.8]
\draw[thick, densely dashed] (0,0)--(1,2)--(2,0);
\draw[thick, dotted, double distance=1.3pt](0,0)--(2,0);

\draw[fill](0,0) circle[radius=0.1];
\draw[fill](2,0) circle[radius=0.1];
\draw[fill](1,2) circle[radius=0.1];

\node[below left] at (0,0) {$x$};
\node[below right] at (2,0) {$z$};
\node[above] at (1,2) {$y$};

\draw[->, thick, decorate, decoration={snake, amplitude=0.4mm, segment length=2mm,  post length=1mm}] (3,1)--(5,1);

\draw[thick, densely dashed] (6,0)--(7,1)--(8,0);
\draw[thick, densely dashed] (7,1)--(7,2);
\draw[thick](6,0)--(8,0);

\draw[fill](6,0) circle[radius=0.1];
\draw[fill](7,1) circle[radius=0.1];
\draw[fill](7,2) circle[radius=0.1];
\draw[fill](8,0) circle[radius=0.1];

\node[below left] at (6,0) {$x$};
\node[below right] at (8,0) {$z$};
\node[right] at (7,1) {$w$};
\node[above] at (7,2) {$y$};

 \node[align=center, below] at  (4,-0.7){\fig{fig:Lambdapaw}: $\Lambda$-paw transformation};
\end{tikzpicture}}
\end{center}

\begin{lemma}\label{lem:Lambda-paw}
Let $\mathcal{G}$ be a non-SPN signed graph, and let $x-y-z$ be a negative path in $\mathcal{G}$.
Then the graph $\widehat{\mathcal{G}}$ obtained by a $\Lambda$-paw operation on $x-y-z$ is a non-SPN signed
graph.
\end{lemma}

\begin{proof}
Let $A\in \cop_n\setminus\spn_n$ be a matrix with signed graph $\mathcal{G}$, and suppose that $x=n-2$, $y=n-1$ and $z=n$. Let $a_{n-2,n-1}=-a$ and $a_{n-1,n}=-b$, where $a,b>0$.
Choose $c>0$ such that $a_{n-2,n}+c>0$. Let
\[E=\left(
              \begin{array}{cccc}
                \frac{ac}{b} & a & c & -\sqrt{\frac{ac}{b}} \\
                a & \frac{ab}{c} & b & -\sqrt{\frac{ab}{c}} \\
                c & b &\frac{bc}{a} & -\sqrt{\frac{bc}{a}} \\
                -\sqrt{\frac{ac}{b}} & -\sqrt{\frac{ab}{c}} & -\sqrt{\frac{bc}{a}} & 1 \\
              \end{array}
            \right).
\]
The matrix $E$ is a rank $1$ positive semidefinite matrix.
Then    $\widehat{A}=(A\oplus0)+(0\oplus E)\in \sym_{n+1}$ is copositive, has signed graph $\widehat{\mathcal{G}}$,
and all the off-diagonal entries in its $(n+1)$-th row and column are nonpositive.
Since $\widehat{A}/\widehat{A}[n+1]=A$ is not SPN, $\widehat{A}$ is also non-SPN.
\end{proof}

\section{Characterizing SPN graphs}
In this section we combine the results to present some sufficient conditions and some necessary conditions for a graph to be SPN.
By Corollary \ref{cor:SPNbyblocks}, it suffices to consider the possible blocks of an SPN graph. Thus we consider below some $2$-connected graphs
(graphs with at least $3$ vertices, and no cut vertex).

\begin{theorem}\label{thm:cycle}
For every $n\ge 3$ the cycle $C_n$ on $n$ vertices is an SPN graph.
\end{theorem}

\begin{proof}
A cycle  on $3$ or $4$ vertices  is SPN by Theorem \ref{thm:n<5}. The graph $C_4$ contains an independent path of length $3$, and $C_n$ is obtained  from $C_4$ by subdividing one of these path's edges  ($n-4$ times), and is therefore SPN by Lemma \ref{lem:P4}.
\end{proof}

\begin{theorem}\label{thm:subdivided diamond}
Any graph which is a subdivision of the diamond graph $T_4$ (by replacing none, any, some, or all the diamond edges by independent paths)
is SPN.
\end{theorem}

\begin{proof}
Let $G$ be a graph on $n$ vertices, which is a subdivision of the diamond. It consists of three independent paths,
has two vertices of degree $3$, $n-2$ vertices of degree $2$ and $n+1$ edges.
Such $G$ is $2$-connected, and each of its proper subgraphs  is SPN by Corollary \ref{cor:SPNbyblocks}, Theorem \ref{thm:cycle} and
Theorem \ref{thm:n<5}. Thus, it suffices to show that any $A\in \cop_n$ with $G(A)=G$ and a connected $G_-(A)$ is SPN.
If $G_-(A)$ is connected, it has at least $n-1$ edges, and thus $G_+(A)$ has at most $2$ edges, which means that $A$ has positive off-diagonal elements in at most four rows.
Thus if $\alpha$ is the set of indices of rows of $A$ with all off-diagonal elements nonpositive, then $|\alpha|\ge n-4$. The matrix $A[\alpha]$ is an $M$-matrix, and by
Lemma \ref{lem:M+-rows} (and Remark \ref{rem:singular M} following it) $A$ is SPN,
since $A/A[\alpha]$ is SPN, as a copositive matrix of order at most $4$.
\end{proof}

For $n\ge 5$ we do not know whether a subdivision of $T_n$ is SPN, except for the following case.

\begin{theorem}\label{thm:K2,n}
For every $n$, the complete bipartite graph $K_{2,n}$ is SPN.
\end{theorem}

\begin{proof}
By induction on $n$. For $n\le 2$ this holds by Theorem \ref{thm:n<5}. Suppose $K_{2,n-1}$ is SPN, $n\ge 3$. Then
every proper subgraph of $K_{2,n}$ is SPN by the induction hypothesis, Lemma \ref{lem:subgraph} and Corollary \ref{cor:SPNbyblocks}.
Thus we only need to consider the case that $A\in \cop$ has a connected $G_-(A)$. In this case, there exists a vertex $i$ of degree $2$
in $\mathcal{G}(A)$, which is incident with two negative edges. By Lemma \ref{lem:+row-row} $A/A[i]$ is copositive and since $G(A/A[i])=T_{n+1}$
it is SPN. Thus $A$ is SPN.
\end{proof}

We will see later  that not every bipartite graph is SPN.
Next we present some non-SPN $2$-connected graphs.

\begin{theorem}\label{thm:F5subdivisions}
Any subdivision of $F_5$, obtained by replacing none, any, some, or all of the edges in $F_5$ by independent paths is not SPN.
\end{theorem}

\begin{proof}
By the proof of Lemma \ref{lem:F5}, the  signed $F_5$ shown on the left in Fig. \ref{fig:LampawF5} is not SPN, and therefore
the subgraph obtained from it by $\Lambda$-paw transformation  on $1-2-3$ (shown in Fig. \ref{fig:LampawF5} on the right) is also non-SPN.

\begin{center}{
\begin{tikzpicture}[scale=0.8]
\draw[thick, densely dashed] (0,1)--(1, 0)--(1.75,2.5)--(2.5,0)--(3.5,1);
\draw[thick] (1,0)--(2.5,0);
\draw[thick]  (0,1)--(1.75,2.5);
\draw[thick] (3.5,1)--(1.75,2.5);

\draw[fill](0,1) circle[radius=0.1];
\draw[fill](1,0) circle[radius=0.1];
\draw[fill](2.5,0) circle[radius=0.1];
\draw[fill](3.5,1) circle[radius=0.1];
\draw[fill](1.75,2.5) circle[radius=0.1];

\node[left] at (-0.1,1) {1};
\node[below] at (1,-0.1) {2};
\node[below] at (2.5,-0.1) {4};
\node[right] at (3.6,1) {5};
\node[above] at (1.75,2.6) {3};

\draw[->, thick, decorate, decoration={snake, amplitude=0.4mm, segment length=2mm,  post length=1mm}](5.5,1)--(8.5,1)  node [above, text width=5.2cm, text centered, midway] at
(7,1.2) {$\Lambda$-paw transformation \\on $1-2-3$}    ;


\draw[thick, densely dashed] (10.5,1)--(11.5, 0)--(12.25,2.5)--(13,0)--(14,1);
\draw[thick, densely dashed] (11.5,0)--(12.25,-.75);
\draw[thick] (12.25,-.75)--(13,0);
\draw[thick]  (10.5,1)--(12.25,2.5);
\draw[thick] (14,1)--(12.25,2.5);

\draw[fill](10.5,1) circle[radius=0.1];
\draw[fill](11.5,0) circle[radius=0.1];
\draw[fill](13,0) circle[radius=0.1];
\draw[fill](14,1) circle[radius=0.1];
\draw[fill](12.25,2.5) circle[radius=0.1];
\draw[fill](12.25,-.75) circle[radius=0.1];

\node[left] at (10.4,1) {1};
\node[below] at (11.5,-0.1) {6};
\node[below] at (13,-0.1) {4};
\node[right] at (14.1,1) {5};
\node[above] at (12.25,2.6) {3};
\node[below] at (12.25,-.76){2};

\node[align=center, below] at  (7,-1.4){\fig{fig:LampawF5}: a $\Lambda$-paw transformation on a signed $F_5$};
\end{tikzpicture}}
\end{center}

Let a graph $G$ be a subdivision of $F_5$. Then there is a sign assignment to $G$'s edges that yields a
signed graph, which is obtained from of one of the two graphs shown in Fig. \ref{fig:LampawF5} by subdividing a negative edge. Since both
signed graphs on Fig. \ref{fig:LampawF5} are non-SPN, the graph $G$ is non-SPN.
\end{proof}

Next we present another family of non-SPN graphs.
Let $CD_6$ be the graph shown in Fig. \ref{fig:CD6}.
\begin{center}{
\begin{tikzpicture}[scale=0.8]
\draw[thick] (0,1)--(1, 0)--(3,0)--(4,1)--(3,2)--(1,2)--cycle;
\draw[thick] (1,0)--(1,2);
\draw[thick]  (3,0)--(3,2);

\draw[fill](0,1) circle[radius=0.1];
\draw[fill](1,0) circle[radius=0.1];
\draw[fill](3,0) circle[radius=0.1];
\draw[fill](4,1) circle[radius=0.1];
\draw[fill](3,2) circle[radius=0.1];
\draw[fill](1,2) circle[radius=0.1];

\node[left] at (0,1) {1};
\node[above] at (1,2) {2};
\node[below] at (1,0) {3};
\node[below] at (3,0) {4};
\node[above] at (3,2) {5};
\node[right] at (4,1) {6};

\node[align=center, below] at  (2,-.7){\fig{fig:CD6}: $CD_6$};
\end{tikzpicture}}
\end{center}

\begin{theorem}\label{thm:CD6 subdivisions}
Any subdivision of $CD_6$, obtained by replacing none, any, some, or all of the edges in $CD_6$ by independent paths, is not SPN.
\end{theorem}

\begin{proof}
We first show that $CD_6$ itself is not SPN. Let $t=\frac{\sqrt{2}}{2}$, and
\[A=\left(
      \begin{array}{cccccc}
        1 & -1 & t & 0 & 0 & 0 \\
        -1 & 1 & -t & 0 & 1 & 0 \\
        t & -t & 1 & -t & 0 & 0 \\
        0 & 0 & -t & 1 & -t & t \\
        0 & 1 & 0 & -t & 1 & -1 \\
        0 & 0 & 0 & t & -1 & 1 \\
      \end{array}
    \right).
\]
The matrix $A(5)$ is copositive, since it is the positive semidefinite matrix $A[1,2,3,4]$  bordered by nonnegative column and row.
Similarly, $A(2)$ is copositive. Every principal submatrix of $A$ whose entries are less than $1$ is a principal submatrix of
one of these $5\times 5$ matrices, and is therefore copositive. By Lemma \ref{lem:diag1 cop}, $A$ is copositive.

We now show that $A$ is not SPN. Each of the submatrices $A[i,i+1,i+2]$, $i=1,\dots, 4$, is an $\widetilde{\mathcal{N}}$-irreducible positive semidefinite
matrix. Thus if $A\ge P$, where $P$ is positive semidefinite and $\diag A =\diag P $, then
\[P=\left(
      \begin{array}{cccccc}
        1 & -1 & t & * & * & * \\
        -1 & 1 & -t & 0 & * & * \\
        t & -t & 1 & -t & 0 & * \\
        * & 0 & -t & 1 & -t & t \\
        * & * & 0 & -t & 1 & -1 \\
        * & * & * & t & -1 & 1 \\
      \end{array}
    \right).
\] Since $(1,1,0,0,0)^T\in \mathcal{V}^P$ and $P$ is positive semidefinite, we have that $p_{1j}=-p_{2j}$ (and $p_{j1}=-p_{j2}$) for every $j=2,\dots, 6$.
Similarly, since $(0,0,0,1,1)^T\in \mathcal{V}^P$ and $P$ is positive semidefinite, $p_{i6}=-p_{i5}$ (and $p_{6i}=-p_{5i}$) for every $i=2,\dots, 6$.
We also have    \cob{$(0,t,1,t,0,0)^T\in \mathcal{V}^P$}, which implies that $p_{52}\cdot t+0\cdot 1+(-t)\cdot t=0$, that is, $p_{25}=p_{52}=t>0$.   
By the equalities above,  $p_{15}=p_{26}=-t$ and then $p_{16}=t$. This contradicts the assumption that $A\ge P$. Hence $A$ is not SPN.

In Fig. \ref{fig:LampawCD6} the non-SPN signed graph of the above matrix is shown on the left. On the right is the graph obtained from this graph by $\Lambda$-paw
transformation on $4-5-6$, which is also non-SPN  by Lemma \ref{lem:Lambda-paw}.
\begin{center}{
\begin{tikzpicture}[scale=0.8]
\draw[thick, densely dashed] (0,1)--(1, 0)--(1,2)--(3,2)--(3,0)--(4,1);
\draw[thick] (1,0)--(3,0);
\draw[thick]  (4,1)--(3,2);
\draw[thick] (0,1)--(1,2);

\draw[fill](0,1) circle[radius=0.1];
\draw[fill](1,0) circle[radius=0.1];
\draw[fill](3,0) circle[radius=0.1];
\draw[fill](4,1) circle[radius=0.1];
\draw[fill](3,2) circle[radius=0.1];
\draw[fill](1,2) circle[radius=0.1];

\node[left] at (0,1) {1};
\node[below] at (1,0) {2};
\node[below] at (3,0) {5};
\node[right] at (4,1) {6};
\node[above] at (1,2) {3};
\node[right] at (3,2) {4};

\draw[->, thick, decorate, decoration={snake, amplitude=0.4mm, segment length=2mm,  post length=1mm}](5.5,1)--(8.5,1)  node [above, text width=5.2cm, text centered, midway] at
(7,1.2) {$\Lambda$-paw transformation \\on $4-5-6$}    ;

\draw[thick, densely dashed] (10,1)--(11, 0)--(11,2)--(13,2)--(13,0)--(14,1);
\draw[thick] (11,0)--(12,0);
\draw[thick, densely dashed] (12,0)--(13,0);
\draw[thick]  (14,1)--(13,2);
\draw[thick] (10,1)--(11,2);

\draw[fill](10,1) circle[radius=0.1];
\draw[fill](11,0) circle[radius=0.1];
\draw[fill](12,0) circle[radius=0.1];
\draw[fill](13,0) circle[radius=0.1];
\draw[fill](14,1) circle[radius=0.1];
\draw[fill](13,2) circle[radius=0.1];
\draw[fill](11,2) circle[radius=0.1];

\node[left] at (10,1) {1};
\node[below] at (11,0) {2};
\node[below] at (13,0) {7};
\node[right] at (14,1) {6};
\node[above] at (11,2) {3};
\node[right] at (13,2) {4};
\node[below] at (12,0){5};

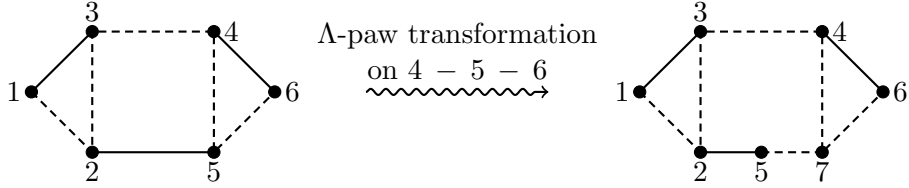
\node[align=center, below] at  (7,-1.4){\fig{fig:LampawCD6}: a $\Lambda$-paw transformation on a signed  $CD_6$};
\end{tikzpicture}}
\end{center}

If $G$ is a  subdivision of $CD_6$, signs can be assigned to its edges so that the resulting signed graph $\mathcal{G}$ is
  obtained from one of the signed graphs shown in   Fig. \ref{fig:LampawCD6} by replacing negative edges by negative paths. Hence  such $G$ is not SPN.
\end{proof}

Next we consider subdivisions of the complete graph on $4$ vertices, $K_4$.
We have already shown that subdividing once an edge of the SPN graph $K_4$ yields the SPN graph $DR_5$. In the next lemmas we
consider further subdivisions of $DR_5$. Denote by $DR_n$ the graph on $n\ge 4$ vertices obtained from $K_4$ by replacing one edge of $K_4$ by an independent path of length $n-3$.

\begin{theorem}\label{thm:DRn}
For every $n\ge 5$, the graph $DR_n$ is an SPN graph.
\end{theorem}

\begin{proof}
For $n=5$ this is  Lemma \ref{lem:DR5}. We prove the result for $n= 6$, and the general case then follows by Lemma \ref{lem:P4}.

Let $A\in \cop_6$ have  $G(A)=DR_6$ (shown on the left in Fig. \ref{fig:DR6signedDR6}). Every proper subgraph of $DR_6$ is SPN, since either it is a subdivision of the diamond $T_4$, or each of its blocks is a subgraph of the diamond. Thus it suffices to
consider the case that $G_-(A)$ is connected. If in the $i$-th row of $A$  all the off-diagonal
entries are nonpositive, then $A/A[i]$ is SPN, as in this case $G(A/A[i])$ is either a subgraph of $DR_5$ or a subgraph of a subdivision of the diamond. Thus $A/A[i]$ is SPN in this case,
and $A$ itself is SPN. So suppose in each row of $A$ there is a positive off-diagonal element ($G_+(A)$ has no isolated vertices). Up to isomorphism, we may assume that
$\mathcal{G}(A)$ is as shown on the right in Fig. \ref{fig:DR6signedDR6}.

\begin{center}
\begin{tikzpicture}[scale=0.8]
\draw[thick] (0, 0)--(2,3)--(4,0)--(2,1.5)--cycle;
\draw[thick] (2,3)--(2,1.5);
\draw[thick] (0,0)--(4,0);

\draw[fill](0,0) circle[radius=0.1];
\draw[fill](4,0) circle[radius=0.1];
\draw[fill](2,1.5) circle[radius=0.1];
\draw[fill](2,3) circle[radius=0.1];
\draw[fill](1.33,0) circle[radius=0.1];
\draw[fill](2.67,0) circle[radius=0.1];


\node[align=center, below] at  (2,-0.7){$DR_6$};

\draw[thick, densely dashed] (9.33,0)--(8, 0)--(10,3)--(10,1.5)--(12,0)--(10.67,0);
\draw[thick] (8,0)--(10,1.5);
\draw[thick] (10,3)--(12,0);
\draw[thick] (9.33,0)--(10.67,0);

\draw[fill](8,0) circle[radius=0.1];
\draw[fill](12,0) circle[radius=0.1];
\draw[fill](10,1.5) circle[radius=0.1];
\draw[fill](10,3) circle[radius=0.1];
\draw[fill](9.33,0) circle[radius=0.1];
\draw[fill](10.67,0) circle[radius=0.1];

\node[left] at (8,0) {2};
\node[right] at (12,0) {5};
\node[below] at (10,1.5) {4};
\node[below] at (10.67,0) {6};
\node[above] at (10,3) {3};
\node[below] at (9.33,0){1};

\node[align=center, below] at  (10,-0.7){a signed $DR_6$};

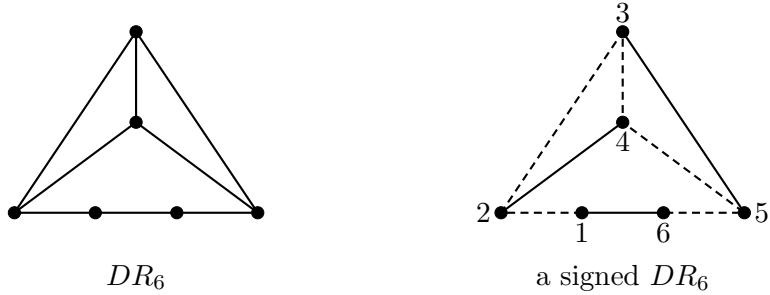
\node[align=center, below] at  (6,-1.4){\fig{fig:DR6signedDR6}: $DR_6$ and a signed $DR_6$};
\end{tikzpicture}
\end{center}

\noindent Then $A$ has the following sign pattern:
\[A=\left(
      \begin{array}{cccccc}
        + & - & 0 & 0 & 0 & + \\
        - & + & - & + & 0 & 0 \\
        0 & - & + & - & + & 0 \\
        0 & + & - & + & - & 0 \\
        0 & 0 & + & - & + & - \\
        + & 0 & 0 & 0 & - & + \\
      \end{array}
    \right).
\]
If the matrix $B$ obtained from $A$ by replacing entry $1,6$ by $0$ is copositive, then $B$ is SPN (since very proper subgraph of $DR_6$ is SPN) and
thus $A$ is also SPN. Otherwise, we may assume that $A$ is $\{1,6\}$-irreducible. Then there exists $\u\in \mathcal{V}^A$ such that $u_1+u_6>0$ and
$(A\u)_1=(A\u)_6=0$. From $(A\u)_1=0$ and $u_1+u_6>0$ we deduce that $u_2>0$, and therefore  $(A\u)_2=0$. Also, if $\v=\left(\begin{array}{cc}
                                                                                                             u_1  &
                                                                                                             u_2
                                                                                                           \end{array}\right)^T$, then  $\left(A[1,2]\v\right)_1\le 0$.
Since $A[1,2]$ is positive definite (as a principal submatrix of the irreducible $M$-matrix $A[1,2,3]$), $\v^TA[1,2]\v>0$, and thus  $(A[1,2]\v)_2> 0$.
Combined with   $(A\u)_2=0$ this implies that $u_3>0$. By the same argument we get from $u_1+u_6>0$ and
$(A\u)_6=0$ that $u_5>0$ and $u_4>0$. That is,  $|\supp\u|\ge 5$ (and $A\u=\0$). Thus $A\in \spn_6$ by Lemma \ref{lem:n-1psd}.
\end{proof}

On the other hand, subdividing two different edges of $K_4$, whether adjacent of not, yields a non-SPN graph.

\begin{lemma}\label{K_4 doubly subdivided}
Let $G$ be a subdivision of $K_4$, obtained by subdividing two edges, once each. Then $G$ is not SPN.
\end{lemma}

\begin{proof}
There are two possible cases: either the two subdivided edges are not adjacent (Case 1), or they are adjacent (Case 2).

\begin{center}
\begin{tikzpicture}[scale=0.8]
\draw[thick] (0, 0)--(2,3)--(4,0)--cycle;
\draw[thick] (0,0)--(2,1.2);
\draw[thick] (2,3)--(2,1.2)--(4,0);

\draw[fill](0,0) circle[radius=0.1];
\draw[fill](4,0) circle[radius=0.1];
\draw[fill](2,1.2) circle[radius=0.1];
\draw[fill](2,3) circle[radius=0.1];
\draw[fill](2,0) circle[radius=0.1];
\draw[fill](2,2.1) circle[radius=0.1];


\node[align=center, below] at  (2,-0.7){Case $1$};

\draw[thick] (8, 0)--(10,3)--(10,1.2)--(12,0)--cycle;
\draw[thick] (8,0)--(10,1.2);
\draw[thick] (10,3)--(12,0);

\draw[fill](8,0) circle[radius=0.1];
\draw[fill](12,0) circle[radius=0.1];
\draw[fill](10,1.2) circle[radius=0.1];
\draw[fill](10,3) circle[radius=0.1];
\draw[fill](9,1.5) circle[radius=0.1];
\draw[fill](11,1.5) circle[radius=0.1];


\node[align=center, below] at  (10,-0.7){Case $2$};

\node[align=center, below] at  (6,-1.4){\fig{fig:K4subd}: Two subdivisions of $K_4$};
\end{tikzpicture}
\end{center}

Case 1: Let $t=\frac{\sqrt{2}}{2}$, and
\[A=\left(\begin{array}{cccccc}
              1 & -t & 0 & 1 & 0 & 0 \\
              -t & 1 & -t & 0 & 1 & 0 \\
              0 & -t & 1 & -t & 0 & 1 \\
              1 & 0 & -t & 1 & -t & 0 \\
              0 & 1 & 0 & -t & 1 & -t \\
              0 & 0 & 1 & 0 & -t & 1 \\
            \end{array}
          \right).\]
The matrix $A$ is copositive, since every principal submatrix in which the off-diagonal entries are less than $1$ is copositive (actually, positive semidefinite).
For every $i=1,\dots, 4$ the submatrix $A[i,i+1,i+2]$ is an $\widetilde{\mathcal{N}}$-irreducible (positive semidefinite) matrix. Thus if $A\ge P$, $P$ positive semidefinite
and $\diag P =\diag A $, then  \[P=\left(\begin{array}{cccccc}
              1 & -t & 0 & * & * & * \\
              -t & 1 & -t & 0 & * & * \\
              0 & -t & 1 & -t & 0 & * \\
              * & 0 & -t & 1 & -t & 0 \\
              * &* & 0 & -t & 1 & -t \\
              * & * & * & 0 & -t & 1 \\
            \end{array}
          \right).\]
Then $\u=(t,1,t,0,0)^T \in \mathcal{V}^P$ implies $P\u=\0$. Thus $p_{41}=p_{14}=t$. Similarly, from $\v=(0,t,1,t,0)^T \in \mathcal{V}^P$ and $\w=(0,0,t,1,t)^T\in
\mathcal{V}^P$ we deduce that $p_{52}=p_{25}=t$ and $p_{63}=p_{36}=t$. But then $P\u=\0$ and $P\v=\0$ imply that $p_{51}=p_{15}=-1$, and $p_{62}=p_{26}=-1$.
Finally, $P\w=\0$ implies now that $p_{61}=p_{16}=t^2>0$, contradicting the assumption that $A\ge P$. Thus $A$ is not SPN.

Case 2: A signed  subdivision of $K_4$  in which two adjacent vertices were subdivided (once each) can be obtained from a signed $F_5$ by
a $\Lambda$-paw transformation. This is shown in Fig. \ref{fig:2ndLampawF5}.

\begin{center}{
\begin{tikzpicture}[scale=0.8]
\draw[thick, densely dashed] (0,1)--(1, 0)--(1.75,2.5)--(2.5,0)--(3.5,1);
\draw[thick] (1,0)--(2.5,0);
\draw[thick]  (0,1)--(1.75,2.5);
\draw[thick] (3.5,1)--(1.75,2.5);

\draw[fill](0,1) circle[radius=0.1];
\draw[fill](1,0) circle[radius=0.1];
\draw[fill](2.5,0) circle[radius=0.1];
\draw[fill](3.5,1) circle[radius=0.1];
\draw[fill](1.75,2.5) circle[radius=0.1];

\node[left] at (-0.1,1) {1};
\node[below] at (1,-0.1) {2};
\node[below] at (2.5,-0.1) {4};
\node[right] at (3.6,1) {5};
\node[above] at (1.75,2.6) {3};

\draw[->, thick, decorate, decoration={snake, amplitude=0.4mm, segment length=2mm,  post length=1mm}](5.5,1)--(8.5,1)  node [above, text width=5.2cm, text centered, midway] at
(7,1.2) {$\Lambda$-paw transformation \\on $2-3-4$}    ;

\draw[thick, densely dashed] (10.5,1)--(11.5, 0)--(12.25,1.1)--(13,0)--(14,1);
\draw[thick, densely dashed] (12.25,1.1)--(12.25,2.5);
\draw[thick] (11.5,0)--(13,0);
\draw[thick]  (10.5,1)--(12.25,2.5);
\draw[thick] (14,1)--(12.25,2.5);

\draw[fill](10.5,1) circle[radius=0.1];
\draw[fill](11.5,0) circle[radius=0.1];
\draw[fill](13,0) circle[radius=0.1];
\draw[fill](14,1) circle[radius=0.1];
\draw[fill](12.25,2.5) circle[radius=0.1];
\draw[fill](12.25,1.1) circle[radius=0.1];

\node[left] at (10.4,1) {1};
\node[below] at (11.5,-0.1) {2};
\node[below] at (13,-0.1) {4};
\node[right] at (14.1,1) {5};
\node[above] at (12.25,2.6) {3};
\node[below] at (12.25,1){6};

\node[align=center, below] at  (7,-1.4){\fig{fig:2ndLampawF5}: Another $\Lambda$-paw transformation on $F_5$};
\end{tikzpicture}}
\end{center}
This signed $F_5$ is  non-SPN by the proof of Lemma \ref{lem:F5}. Thus the graph obtained by subdividing two adjacent edges  of $K_4$, once each, is not SPN.
\end{proof}

\begin{theorem}\label{thm:nonSPN K4subd}
Any graph obtained by subdividing at least two edges and at most five edges of $K_4$, \cob{each of them at least once,}
is not SPN.
\end{theorem}

\begin{proof}
For each such graph, signs can be assigned to the edges, to obtain a signed graph which is a subdivision  of one of the two $K_4$ subdivisions considered in Lemma \ref{K_4 doubly subdivided}. Thus every such graph is non-SPN. (In fact, the signed graph
is either the graph shown on the left in Fig. \ref{fig:nonSPNsbdK4}, or it is a subdivision of the graph shown on the  in that figure.)

\begin{center}{
\begin{tikzpicture}[scale=0.8]
\draw[thick, densely dashed] (2,0)--(0, 0)--(2,3)--(4,0)--(2,1.2)--(2,2.1) ;
\draw[thick] (0,0)--(2,1.2);
\draw[thick] (2,0)--(4,0);
\draw[thick] (2,3)--(2,2.1);

\draw[fill](0,0) circle[radius=0.1];
\draw[fill](4,0) circle[radius=0.1];
\draw[fill](2,1.2) circle[radius=0.1];
\draw[fill](2,3) circle[radius=0.1];
\draw[fill](2,0) circle[radius=0.1];
\draw[fill](2,2.1) circle[radius=0.1];

\draw[thick, densely dashed] (10.5,1.5)--(11.5, 0)--(12.5,1.5)--(13.5,0)--(14.5,1.5);
\draw[thick, densely dashed] (12.5,1.5)--(12.5,3);
\draw[thick] (11.5,0)--(13.5,0);
\draw[thick]  (10.5,1.5)--(12.5,3);
\draw[thick] (14.5,1.5)--(12.5,3);

\draw[fill](10.5,1.5) circle[radius=0.1];
\draw[fill](11.5,0) circle[radius=0.1];
\draw[fill](13.5,0) circle[radius=0.1];
\draw[fill](14.5,1.5) circle[radius=0.1];
\draw[fill](12.5,3) circle[radius=0.1];
\draw[fill](12.5,1.5) circle[radius=0.1];

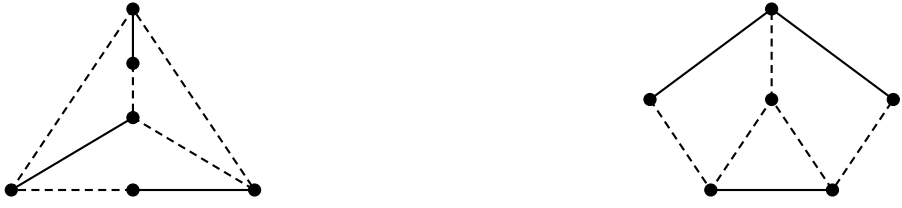
\node[align=center, below] at  (7,-1.4){\fig{fig:nonSPNsbdK4}: Non-SPN subdivisions of a signed $K_4$};
\end{tikzpicture}} \end{center}\qedhere
\end{proof}

Note that the graph shown on the left in Fig. \ref{fig:nonSPNsbdK4} is an example of a non-SPN bipartite graph: it is $K_{3,3}-e$ (where $e$ is any edge of $K_{3,3}$).
Every proper subgraph of $K_{3,3}-e$ is SPN  and every bipartite graph with one independent set of size $2$ is SPN, so in that sense $K_{3,3}-e$ is the smallest non-SPN bipartite graph.

From the above results, we get sufficient conditions for a graph to be SPN, in terms of its possible blocks.

\begin{theorem}\label{thm:sufficient}
Let $G$ be a graph, in which each block is one of the following:
\begin{enumerate}
\item[{\rm(a)}] an edge,
\item[{\rm(b)}] a cycle,
\item[{\rm(c)}] a $T_n$,
\item[{\rm(d)}] a $K_{2,n}$,
\item[{\rm(e)}] a subdivision of the diamond $T_4$,
\item[{\rm(f)}] a subdivision of $K_4$, where at most one edge was subdivided.
\end{enumerate}
Then $G$ is SPN.
\end{theorem}

Note that the graphs in (c), (d), (e) are all $T_n$ subdivisions for some $n\ge 4$, and a cycle can be viewed as a subdivision of $T_3$. Also,  the graphs described in (f) are either a $DR_n$ or a $K_4$ (which can be viewed as $DR_4$).

We also get some necessary conditions for a graph to be SPN, in terms of forbidden subgraphs.

\begin{theorem}\label{thm:necessary}
Let $G$ be an SPN graph. Then $G$ does not contain the following subgraphs:
\begin{enumerate}
\item[$(1)$] a subdivision of the fan $F_5$,
\item[$(2)$] a subdivision of $CD_6$,
\item[$(3)$] a subdivision of $K_4$, where at least two edges and at most five edges were subdivided, \cob{each at least once}.
\end{enumerate}
\end{theorem}

\begin{remark}\label{rem:suff-necc gap}{\rm Theorems \ref{thm:sufficient} and \ref{thm:necessary} do not give a full characterization of SPN graphs. So
what is the next step?

If $G$ is $2$-connected  and $H_0$ is a $2$-connected subgraph of $G$, then $G$ can be generated from $H_0$ by successively adding an $H$-paths to graphs  $H$ already constructed (proof similar to that of \cite[Proposition 3.1.2]{Diestel2010}).
Suppose $G$ is a $2$-connected graph, which does not contain a subdivision of $F_5$ or a subdivision of $CD_6$.
Let $k$ be maximal such that $G$ contains a subdivision of $T_k$. (Recall that $T_3$ is the triangle, and $T_4$ is the diamond.)

If $k=3$, $G$ contains a cycle $H_0$, and since it does not contain a subdivision of $T_4$, it contains no $H_0$-path, so $G=H_0$ is a cycle.

If $k\ge 5$, $G$ contains a subdivision $H_0$ of $T_k$. There is no $H_0$-path in $G$ whose ends are the base vertices of $H_0$,
by the maximality of $k$, and there is no $H_0$-path in $G$ both of whose ends lie on the same independent path in $H_0$, since  $G$ does not contain a subdivision
of $F_5$ or of $CD_6$. There is no $H_0$-path in $G$ one of whose ends lie on one independent path in $H_0$, and the other on another, since $G$ does not contain a subdivision of $F_5$.
So also in this case $G=H_0$ is a subdivision of $T_k$, $k\ge 5$.

If $k= 4$, $G$ contains a subdivision $H_0$ of the diamond $T_4$. That is, $H_0$ consists of $3$ independent paths sharing the same end vertices $x$ and $y$. If $G\ne H_0$, then it contains an $H_0$-path. The ends of this path cannot both be $x$ and $y$, by the maximality of $k$. The ends of the path cannot both lie on one of the three
$xy$-paths  in $H_0$, since $G$ does not contain a subdivision of $F_5$ or of a $CD_6$. Thus one end of the path is an inner vertex of one $xy$-path, and the other is an inner
vertex of another  such path. That is, $G$ contains a subdivision $H$ of $K_4$. If $G\ne H$, then there is an additional $H$-path. Similar to the
previous argument, the end vertices of this $H$-path should be inner vertices in two independent paths in $H$ that do not share end vertices.

Therefore, if a graph $G$ contains no subdivision of $F_5$, no subdivision of $CD_6$, and no subdivision of $K_4$ where at least two edges have been subdivided,
then $G$ is either an edge, or a subdivision of $T_k$, $k\ge 3$, or a $DR_k$, $k\ge 4$.
To complete the characterization of SPN graphs, one first has to check whether (or which) subdivisions of $T_n$, $n\ge 5$, are SPN,
and whether (or which) subdivisions of $K_4$,  in which all six edges where subdivided, are SPN. If some of the latter turn out to be SPN, more graphs need to be
checked.}
\end{remark}

In view of the last remark, we make the following conjecture.

\begin{conjecture}\label{con:subd Tn}
Every subdivision of $T_n$ is SPN.
\end{conjecture}

As for the remaining subdivisions of $K_4$:
Let $G$ be the subdivision of $K_4$ where each of the six edges is subdivided once. We do not know if $G$ is SPN. By subdividing   negative edges in
the corresponding non-SPN signed graph
it can be seen that if this $G$ is not SPN, then any further
subdivision of edges of $G$ yields a non-SPN graph. However, if this $G$ is SPN, then its subdivisions may or may not be SPN. Our next conjecture is the following.

\begin{conjecture}\label{con:subd K4}
Every subdivision of $K_4$  in which all six edges were subdivided, \cob{each at least once,} is SPN.
\end{conjecture}

If Conjecture \ref{con:subd K4} is true, there are more graphs that may be SPN. Taking a step further, our final conjecture is the following.

\begin{conjecture}\label{con:SPN characterization}
The list of  graphs in Theorem \ref{thm:necessary} is a complete list of forbidden subgraphs for the property of being SPN.
\end{conjecture}

\end{document}